\documentclass{svjour3}                     % onecolumn (standard format)
\smartqed  % flush right qed marks, e.g. at end of proof
\usepackage{bbm}
\usepackage{mathrsfs}
\usepackage{amsfonts}
\usepackage{stmaryrd}
\usepackage{graphicx}
\usepackage{subfigure}

\usepackage{bm}% bold math
\usepackage{amsmath}
\usepackage{amssymb}

\usepackage{amstext}
\usepackage{amsbsy}
\usepackage{setspace}
\usepackage{algorithm}
\usepackage{algorithmic}
\usepackage[figuresright]{rotating}
\usepackage{epstopdf}
\usepackage{makecell}
\usepackage{color}

\numberwithin{equation}{section}
\numberwithin{theorem}{section}
\numberwithin{lemma}{section}
\numberwithin{remark}{section}
\usepackage{mathptmx}

\newtheorem{ass}{Assumption}

\usepackage{mathptmx}      % use Times fonts if available on your TeX system
\usepackage[justification=centering]{caption}
\date{}
\begin{document}

\title{Difference methods for time discretization of
stochastic wave equation}


\author{Xing Liu
}


\institute{
Xing Liu \at
School of Mathematics and Economics, Bigdata Modeling and Intelligent Computing research institute, Hubei University of Education, Wuhan 430205, People's Republic of China. \email{2718826413@qq.com}
             % Tel.: +123-45-678910\\
%              Fax: +123-45-678910\\
             % \email{wsqlzu@gmail.com}
 %  \emph{Present address:} of F. Author  %  if needed
%\and
%Weihua Deng\at School of Mathematics and Statistics, Gansu Key Laboratory of Applied Mathematics and Complex
%Systems, Lanzhou University, Lanzhou 730000, People's Republic of China.
%\email{dengwh@lzu.edu.cn}
}

\maketitle

\begin{abstract}	
The time discretization of stochastic spectral fractional wave equation is studied by using the difference methods. Firstly, we exploit rectangle formula to get a low order time discretization, whose the strong convergence order is smaller than $1$ in the sense of mean-squared $L^2$-norm. Meanwhile, by modifying the low order method with trapezoidal rule, the convergence rate is improved at expenses of requiring some extra temporal regularity to the solution. The modified scheme has superlinear convergence rate under the mean-squared $L^2$-norm. Several numerical experiments are provided to confirm the theoretical error estimates.
\\
\\
%\noindent{\bf AMS} 26A33, 35R11, 65M60, 65M12.

\end{abstract}

\section{Introduction}
\label{sec:into}
Nonlocal operators are successful in trying to explain phenomena and fit data in complex systems. Thus, fractional operators become of high interest in modeling the wave propagations (e.g., viscous damping in the seismic isolation of buildings, cardiac electrical propagation, and seismic wave propagation \cite{Bueno,ChenHolm,MeerschaertSchilling,Szabo}). The work presented in this manuscript is based on space-fractional operator on a bounded domain. Moreover, we are also concerned with the external noises that possibly affect the wave propagation. With the above introduction, the model we discuss in this paper is the stochastic wave equation (SWE)
\begin{equation}\label{eq:1.05}
\left\{
\begin{array}{ll}
\frac{\mathrm{d} \dot{u}(x,t)}{\mathrm{d} t}
=-(-\mathrm{\Delta})^\alpha u(x,t)+f\left(u(x,t)\right)+\dot{B}_H(x,t) \quad \mathrm{in} \ D\times(0,T],\\[1.5mm]
u(x,0)=u_0, \  \dot{u}(x,0)=v_0  %(\partial u(x,t)/\partial t)|_{t=0}=v_0
\quad  \mathrm{in}\ D,\\[1.5mm]
u(x,t)=0 \quad \mathrm{in}\ \partial D\times(0,T],
\end{array}
\right.
\end{equation}
where $\dot{u}(x,t)$ is the first order time derivative of $u(x,t)$, $d/dt$ means the partial derivative with respect to $t$, $f$ is the source term, $D\subset \mathbb{R}^d\ (d=1,2,3)$, and $\dot{B}_H(x,t)$ is the infinite dimensional space-time fractional Gaussian noise, $(-\mathrm{\Delta})^\alpha$ is spectral fractional Laplacian \cite{Nochetto, Song} with $0<\alpha\le1$ and $\frac{1}{2}<H<1$. Let $-\mathrm{\Delta}$ be Laplacian on a bounded domain $D$. If $\left\{\lambda_i,\phi_i(x)\right\}^\infty_{i=1}$ is the eigenpairs of $-\mathrm{\Delta}$, then
\begin{equation}\label{eq:1.06}
\left\{
\begin{array}{ll}
-\mathrm{\Delta}\phi_i(x)=\lambda_i\phi_i(x),\quad  \mathrm{in}\ D,\\
~~~~~\phi_i(x)=0, ~~~~~~~~\quad \mathrm{on}\ \partial D,
\end{array}
\right.
\end{equation}
and
\begin{equation}\label{eq:1.07}
\left\{
\begin{array}{ll}
(-\mathrm{\Delta})^\alpha\phi_i(x)=\lambda^\alpha_i\phi_i(x),\quad  \mathrm{in}\ D,\\
~~~~~~~~~\phi_i(x)=0, ~~~~~~~~~\quad \mathrm{on}\ \partial D.
\end{array}
\right.
\end{equation}

Generally, the exact solution intended for the realistic application is unknown for Eq. \eqref{eq:1.05}. Thus, finding an approximation of Eq. \eqref{eq:1.05} became of high interest in mathematical computing. Many numerical schemes have been studied to solve SWE with additive or multiplicative noise. For instance, the spectral Galerkin method has been proposed for the spatial discretization of SWE driven by additive noise in \cite{LiuDeng,WangGan}. The authors of \cite{CaoHong,Cohen,Kovacs,LiWang} proposed a finite element element technique to compute the solution of SWE with additive noise. In a series of works on numerical analysis of SWE forced by multiplicative noise, the finite difference method and finite element method has been used respectively to obtain space approximation in \cite{Anton} and \cite{CohenQuer}. %the fully discrete approximation have been presented.
Popularly, the temporal approximation of SWE has been designed by mild solution formulation with the boundary condition (e.g., stochastic
trigonometric method \cite{Anton,Cohen,CohenQuer}, exponential time integrators \cite{WangGan}, rational approximation \cite{Kovacs,KovacsLar}). In \cite{LiuDeng}, a superlinear convergence rate was obtained in time by modifying stochastic trigonometric method when $\dot{u}(x,t)$ is bounded in the sense of mean-squared $L^2$-norm.

Differently from the methods mentioned above, we will provide here difference methods to get time approximation of problem \eqref{eq:1.05}. To dealt with the non H\"older continuity of $\dot{u}(x,t)$ in some cases \cite{Liu}, the system \eqref{eq:1.05} is firstly transformed into an equivalent form by using Ornstein-Uhlenbeck process, i.e.,
\begin{equation*}
\left\{
\begin{array}{ll}
\frac{\mathrm{d} \dot{z}(x,t)}{\mathrm{d} t}
=-(-\mathrm{\Delta})^\alpha z(x,t)+f\left(u(x,t)\right)\quad \mathrm{in} \ D\times(0,T],\\[1.5mm]
z(x,0)=u_0, \  \dot{z}(x,0)=v_0
\quad  \mathrm{in}\ D,\\[1.5mm]
z(x,t)=0 \quad \mathrm{in}\ \partial D\times(0,T],
\end{array}
\right.
\end{equation*}
where
\begin{equation*}
u(x,t)=z(x,t)+\int^{t}_0A^{-\frac{\alpha}{2}}\sin\left((-\mathrm{\Delta})^{\frac{\alpha}{2}}(t-s)\right)\mathrm{d}B_H(x,s).
\end{equation*}
Then we design only a low order time discretization of \eqref{eq:1.05} by using above equation and rectangle formula, because the convergence rate of difference method is limited by the smoothness of the solution. Meanwhile, a higher order time discretization is proposed based on trapezoidal rule and Lagrange mean value theorem, when $\dot{u}(x,t)$ is time H\"older continuous in the sense of mean-squared $L^2$-norm.

This paper is organized as follows. In the next section, we introduce some notations and preliminaries, including assumptions and concepts of fractional Brownian motion (fBm). After proving our main regularity estimates, in Section \ref{sec:3}, based on the Dirichlet eigenpairs and concepts of fBm, we present respectively the time regularity estimates of $u(x,t)$ and $z(x,t)$ in the sense of mean-squared $L^2$-norm. In Section \ref{sec:4}, we propose two strategies to obtain the time-discretization with order of convergence depending on the time regularity of $u(t)$. And stability and convergence analysis of two shcemes is given. The numerical experiments are performed in Section \ref{sec:6}. We end the paper with some discussions in Section \ref{sec:7}.

\section{Notations and Preliminaries} \label{sec:2}

In this section, we define functional spaces and gather preliminary results on the Dirichlet eigenpairs and fBm, which are commonly used in the rest of this paper.

Let $U=L^2(D;\mathbb{R})$ be a real separable Hilbert space. And we will always use $\|\cdot\|$ to indicate the norm in $U$. We define the unbounded linear operator $A^\nu$ by $A^\nu u=\left(-\mathrm{\Delta}\right)^\nu u$ on the domain
\begin{equation*}
\mathrm{dom}\left(A^{\nu}\right)=\left\{ A^\nu u\in U:u(x)=0,\ x\in \partial D \right\}.
\end{equation*}
There exists an orthonormal basis of $U$ consisting
of eigenfunctions $\{\phi_{i}(x)\}^\infty_{i=1}$ of Laplacian under homogeneous Dirichlet boundary condition. Then
\begin{equation*}
u=\sum^\infty_{i=1}\left\langle u,\phi_{i}(x)\right\rangle\phi_{i}(x)
\end{equation*}
and
\begin{equation*}
A^{\frac{\nu}{2}}u=\sum^\infty_{i=1}\lambda^{\frac{\nu}{2}}_{i}\left\langle u,\phi_{i}(x)\right\rangle\phi_{i}(x),
\end{equation*}
where $\langle\cdot,\cdot\rangle$ denotes the $U$ inner product, and $\lambda^{\frac{\nu}{2}}_{i}$, $i=1,2,\dots,$ are the eigenvalues of spectral fractional Laplacian.
Moreover, we define the Hilbert space $\dot{U}^\nu=\mathrm{dom}\left(A^{\frac{\nu}{2}}\right)$
equipped with the inner product
\begin{equation*}
\left\langle u, v\right\rangle_{\nu}=\sum^\infty_{i=1}\lambda^{\frac{\nu}{2}}_{i}\left\langle u,\phi_{i}(x)\right\rangle\times\lambda^{\frac{\nu}{2}}_{i}\left\langle v,\phi_{i}(x)\right\rangle
\end{equation*}
and norm
\begin{equation*}
\|u\|_\nu=\left(\sum^\infty_{i=1}\lambda^{\nu}_{i}\left\langle u,\phi_{i}(x)\right\rangle^2\right)^{\frac{1}{2}}.
\end{equation*}
In particular, $\dot{U}^0=U$.

\begin{lemma} \label{le:01}
Let $\Omega$ denote a bounded domain in $\mathbb{R}^d$, $d\in\{1,2,3\}$. Let $\lambda_{i}$ be the i-th eigenvalue of the Dirichlet homogeneous boundary problem for the Laplacian operator $-\mathrm{\Delta}$ in $\Omega$. %Let $|\Omega|$ be the volume of $\Omega$, the
Then
\begin{equation*}
C_0i^{\frac{2}{d}}\le\lambda_{i}\le C_1i^{\frac{2}{d}},
\end{equation*}
where $i\in\mathbb{N}$, and the constants $C_0$ and $C_1$ are independent of $i$.
\end{lemma}
The detail proof of Lemma \ref{le:01} has been shown in \cite{Laptev,LiYau,Strauss}.

\begin{ass}\label{as:2.1}
The function $f:U\to U$ in (\ref{eq:1.05}) satisfies
\begin{equation*}
\|f(u)-f(v)\| \lesssim \|u-v\| ~ {\rm for ~any} ~u,v\in U,
\end{equation*}
and
\begin{equation*}
\|A^{\frac{\nu}{2}} f(u)\|\lesssim 1+\|A^{\frac{\nu}{2}} u\|  ~ {\rm for } ~ u\in \dot{U}^\nu~ {\rm with } ~ \nu\in[0,\gamma],
\end{equation*}
where $\gamma$ is given in Lemma \ref{le:1}.
\end{ass}

Next, we recall fBm and corresponding stochastic integral. Let $\beta_H(t)$ be the two-sided one-dimensional fBm, then
$\mathrm{E}[\beta_H(t)]=0$ and
\begin{equation*}
\mathrm{E}[\beta_H(t)\beta_H(s)]=\frac{1}{2}\left(|t|^{2H}+|s|^{2H}-|t-s|^{2H}\right).
\end{equation*}
For $H\in\left(\frac{1}{2},1\right)$ and $\int^t_sf(r)(r-s)^{H-\frac{3}{2}}\left(\frac{s}{r}\right)^{\frac{1}{2}-H}\mathrm{d}r$ belongs to $L^2\left([0,T],\mathbb{R}\right)$, an expression of the Wiener integral with respect to fBm is introduced as \cite{AlosMazet,WangZeng}
\begin{equation}\label{eq:2.1}
\int^t_0f(s)\mathrm{d}\beta_H(s)=C_H\int^t_0\int^t_s\left(H-\frac{1}{2}\right)f(r)(r-s)^{H-\frac{3}{2}}\left(\frac{s}{r}\right)^{\frac{1}{2}-H}\mathrm{d}r\mathrm{d}\beta(s),\ s,t\in[0,T],
\end{equation}
where $\beta(t)$ is standard Brownian motion and
\begin{equation*}
C_H=\left(\frac{2H\times\Gamma(\frac{3}{2}-H)}{\Gamma(H+\frac{1}{2})\Gamma(2-2H)}\right)^{\frac{1}{2}}.
\end{equation*}
The following Lemma shows the integral theory of fBm.
\begin{lemma} [\cite{Mishura}]\label{le:02}
For $f, g\in L^2(\mathbb{R};\mathbb{R})\cap L^1(\mathbb{R};\mathbb{R})$, as $\frac{1}{2}<H<1$, we have
\begin{equation*}
\mathrm{E}\left[\int_Rf(s)\mathrm{d}\beta_H(s)\right]=0
\end{equation*}
and
\begin{equation*}
\mathrm{E}\left[\int_Rf(s)\mathrm{d}\beta_H(s)\int_Rg(t)\mathrm{d}\beta_H(t)\right]=H(2H-1)\int_R\int_Rf(s)g(t)|s-t|^{2H-2}\mathrm{d}s\mathrm{d}t.
\end{equation*}
\end{lemma}

Let driven stochastic process $B_{H}(x,t)$ be a cylindrical fBm with respect to the normal filtration $\{\mathcal{F}_t\}_{t\in[0,T]}$. The
infinite dimensional space-time Gaussian process can be represented by the formal series
\begin{equation*}
B_{H}(x,t)=\sum^\infty_{i=1}\sigma_{i}\beta^{i}_{H}(t)\phi_{i}(x),
\end{equation*}
where $|\sigma_{i}|\lesssim\lambda_i^{-\rho}$ ($\rho\ge0$, $\lambda_i$ is given in Lemma \ref{le:01}), $\beta^{i}_{H}(t)$, $i=1,2,\dots,$ are mutually independent real-valued fractional Brownian motions with $\frac{1}{2}<H<1$, and $\{\phi_{i}(x)\}_{i\in\mathbb{N}}$ is an orthonormal basis of $U$.

We define $L^p(D,\dot{U}^\nu)$ to be the separable Hilbert space of $p$-times integrable random variables with norm
\begin{equation*}
\left\|u\right\|_{L^p(D,\dot{U}^\nu)}=\left(\mathrm{E}\left[\left\|u\right\|^p_{\nu}\right]\right)^{\frac{1}{p}}, \quad \nu\ge0.
\end{equation*}

\section{H\"older Continuity of the Solution} \label{sec:3}
To begin with we can give a system of equations by coupling Eq. \eqref{eq:1.05} and $\mathrm{d}u(x,t)=\dot{u}(x,t)\mathrm{d}t$. Combining the system of equations and Ornstein-Uhlenbeck process leads to an equivalent form of Eq. \eqref{eq:1.05}, which will be used to obtain the approximation
of Eq. \eqref{eq:1.05}. Moreover, we study time regularity of the mild solution $u(x,t)$ and $z(x,t)$, respectively.

In the interest of brevity and readability, replace Eq. \eqref{eq:1.05} with the following equation
\begin{equation}\label{eq:3.1}
\left\{
\begin{array}{ll}
\mathrm{d} \dot{u}(t)=-A^\alpha  u(t)\mathrm{d}t+f\left(u(t)\right)\mathrm{d}t+\mathrm{d}B_H(t)& \quad \mathrm{in} \ D\times(0,T],\\[1mm]
u(0)=u_0, \dot{u}(0)=v_0& \quad  \mathrm{in}\ D,\\[1mm]
u(t)=0& \quad \mathrm{in}\ \partial D,
\end{array}
\right.
\end{equation}
where $u(t)=u(x,t)$ and $B_H(t)=B_H(x,t)$. Let $v(t)=\dot{u}(t)$, then we have
\begin{equation*}
\left\{
\begin{array}{ll}
\mathrm{d} u(t)=v(t)\mathrm{d}t\\
\mathrm{d} v(t)=-A^\alpha  u(t)\mathrm{d}t+f\left(u(t)\right)\mathrm{d}t+\mathrm{d}B_H(t)
\end{array}
\right.,
\end{equation*}
which implies
\begin{equation}\label{eq:3.2}
\mathrm{d}X(t)=\Lambda X(t)\mathrm{d}t+
\left[\begin{array}{c}
0\\
f(u(t))
\end{array} \right]\mathrm{d}t+
\left[\begin{array}{c}
0\\
I
\end{array} \right]\mathrm{d}B_H(t),
\end{equation}
where
\begin{equation*}
X(t)=
\left[\begin{array}{c}
u(t)\\
v(t)
\end{array} \right],\qquad
\Lambda=
\left[\begin{array}{cc}
0& I\\
-A^{\alpha}&0
\end{array}\right].
\end{equation*}
Then a formal mild solution $X(t)$ for \eqref{eq:3.2} is given as
\begin{equation}\label{eq:3.3}
X(t)=\mathrm{e}^{\Lambda t}X(0)+\int^t_0\mathrm{e}^{\Lambda (t-s)}
\left[\begin{array}{c}
0\\
f\left(u(s)\right)
\end{array} \right]\mathrm{d}s+\int^t_0\mathrm{e}^{\Lambda (t-s)}\left[\begin{array}{c}
0\\
I
\end{array} \right]\mathrm{d}B_H(s).
\end{equation}
Here, $\mathrm{e}^{\Lambda t}$ can be expressed as
\begin{equation}\label{eq:3.3-1}
\mathrm{e}^{\Lambda t}=
\left[\begin{array}{cc}
\cos\left(A^{\frac{\alpha}{2}}t\right)&A^{-\frac{\alpha}{2}}\sin\left(A^{\frac{\alpha}{2}}t\right)\\[1mm]
-A^{\frac{\alpha}{2}}\sin\left(A^{\frac{\alpha}{2}}t\right)&\cos\left(A^{\frac{\alpha}{2}}t\right)
\end{array}\right].
\end{equation}
Substituting \eqref{eq:3.3-1} into \eqref{eq:3.3}, then we get two components of $X(t)$
\begin{equation}\label{eq:3.4}
\begin{split}
u(t)=&\cos\left(A^{\frac{\alpha}{2}}t\right)u_0+A^{-\frac{\alpha}{2}}\sin\left(A^{\frac{\alpha}{2}}t\right)v_0+\int^t_0A^{-\frac{\alpha}{2}}\sin\left(A^{\frac{\alpha}{2}}(t-s)\right)f\left(u(s)\right)\mathrm{d}s\\[1.5mm]
&+\int^t_0A^{-\frac{\alpha}{2}}\sin\left(A^{\frac{\alpha}{2}}(t-s)\right)\mathrm{d}B_H(s),\\[1.5mm]
v(t)=&-A^{\frac{\alpha}{2}}\sin\left(A^{\frac{\alpha}{2}}t\right)u_0+\cos\left(A^{\frac{\alpha}{2}}t\right)v_0+\int^t_0\cos\left(A^{\frac{\alpha}{2}}(t-s)\right)f\left(u(s)\right)\mathrm{d}s\\[1.5mm]
&+\int^t_0\cos\left(A^{\frac{\alpha}{2}}(t-s)\right)\mathrm{d}B_H(s).
\end{split}
\end{equation}

Firstly, we consider the regularity estimates of stochastic integral in \eqref{eq:3.4}. Let $\gamma>\alpha$ and $s\le t$. Combining Lemma \ref{le:01}, and Eq. \eqref{eq:2.1} leads to
 \begin{equation}\label{eq:3.5-1}
 \begin{split}
&\mathrm{E}\left[\left\|\int^t_s\cos\left(A^{\frac{\alpha}{2}}(t-r)\right)\mathrm{d}B_H(s)\right\|^2\right]\\
&\lesssim H(2H-1)\sum\limits^\infty_{i=1}\lambda_i^{-2\rho}\int^t_s\int^t_s\cos\left(\lambda_i^{\frac{\alpha}{2}}(t-r)\right)\cos\left(\lambda_i^{\frac{\alpha}{2}}(t-r_1)\right)|r-r_1|^{2H-2}\mathrm{d}r\mathrm{d}r_1\\
&\lesssim H(2H-1)\sum\limits^\infty_{i=1}\lambda_i^{-2\rho}\int^t_s\int^t_s|r-r_1|^{2H-2}\mathrm{d}r\mathrm{d}r_1\\
&= 2H(2H-1)\sum\limits^\infty_{i=1}\lambda_i^{-2\rho}\int^t_s\int^t_{r_1}(r-r_1)^{2H-2}\mathrm{d}r\mathrm{d}r_1\\
&= \sum\limits^\infty_{i=1}\lambda_i^{-2\rho}(t-s)^{2H}
\end{split}
 \end{equation}
and
\begin{equation}\label{eq:3.5-0}
 \begin{split}
&\mathrm{E}\left[\left\|\int^s_0\cos\left(A^{\frac{\alpha}{2}}(t-r)\right)\mathrm{d}B_H(r)-\int^s_0\cos\left(A^{\frac{\alpha}{2}}(s-r)\right)\mathrm{d}B_H(r)\right\|^2\right]\\
&\lesssim H(2H-1)\sum^\infty_{i=1}\lambda_i^{-2\rho}\int^s_0\int^s_0\left[\cos\left(\lambda_i^{\frac{\alpha}{2}}(t-r)\right)-\cos\left(\lambda_i^{\frac{\alpha}{2}}(s-r)\right)\right]\\
&~~~~\times\left[\cos\left(\lambda_i^{\frac{\alpha}{2}}(t-r_1)\right)-\cos\left(\lambda_i^{\frac{\alpha}{2}}(s-r_1)\right)\right]|r-r_1|^{2H-2}\mathrm{d}r\mathrm{d}r_1\\
&\lesssim H(2H-1)\sum^\infty_{i=1}\lambda_i^{-2\rho}\int^s_0\int^s_0\left(\lambda_i^{\frac{\alpha}{2}}(t-s)\right)^{2*\min\{\frac{\gamma-\alpha}{\alpha},1\}}|r-r_1|^{2H-2}\mathrm{d}r\mathrm{d}r_1\\
&\lesssim (t-s)^{2*\min\{\frac{\gamma-\alpha}{\alpha},1\}}s^{2H}\sum\limits^\infty_{i=1}i^{\frac{2(\gamma-\alpha-2\rho)}{d}}
 \end{split}
 \end{equation}

Combining Lemma \ref{le:01}, Eqs. \eqref{eq:3.4}-\eqref{eq:3.5-0}, and the Burkh\"older-Davis-Gundy inequality \cite{PratoZab,vanNeerven}, we can get the following regularity results of the mild solution $u(t)$ and $v(t)$.
\begin{lemma}\label{le:1}
Suppose that Assumptions \ref{as:2.1} are satisfied, $\left\|u_0\right\|_{L^p(D,\dot{U}^\gamma)}<\infty$, $\left\|v_0\right\|_{L^p(D,\dot{U}^{\gamma-\alpha})}<\infty$, $\varepsilon>0$, $\gamma=\alpha+2\rho-\frac{d+\varepsilon}{2}$, $\gamma>0$, and $p>1$.
Then there exists a unique mild solution $X(t)$ for \eqref{eq:3.2} and
\begin{equation}\label{eq:3.6}
\left\|u(t)\right\|_{L^p(D,\dot{U}^\gamma)}+\left\|v(t)\right\|_{L^p(D,\dot{U}^{\gamma-\alpha})}\lesssim \frac{t^H}{\varepsilon}+\left\|u_0\right\|_{L^p(D,\dot{U}^\gamma)}+\left\|v_0\right\|_{L^p(D,\dot{U}^{\gamma-\alpha})}.
\end{equation}
Furthermore,

$\mathrm{(i)}$ for $\gamma\le\alpha$,
\begin{equation*}
\left\|u(t)-u(s)\right\|_{L^2(D,U)}\lesssim (t-s)^{\frac{\gamma}{\alpha}}\left(\frac{t^H}{\varepsilon}+\left\|u_0\right\|_{L^2(D,\dot{U}^\gamma)}+\left\|v_0\right\|_{L^2(D,\dot{U}^{\gamma-\alpha})}\right);
\end{equation*}

$\mathrm{(ii)}$ for $\alpha<\gamma\le(1+H)\alpha$,
\begin{equation*}
\left\|v(t)-v(s)\right\|_{L^2(D,U)}\lesssim (t-s)^{\frac{\gamma-\alpha}{\alpha}}\left(\frac{t^H}{\varepsilon}+\left\|u_0\right\|_{L^2(D,\dot{U}^\gamma)}+\left\|v_0\right\|_{L^2(D,\dot{U}^{\gamma-\alpha})}\right);
\end{equation*}

$\mathrm{(iii)}$ for $\gamma>(1+H)\alpha$,
\begin{equation*}
\left\|v(t)-v(s)\right\|_{L^2(D,U)}\lesssim (t-s)^H\left(t^H+\left\|u_0\right\|_{L^2(D,\dot{U}^\gamma)}+\left\|v_0\right\|_{L^2(D,\dot{U}^{\gamma-\alpha})}\right).
\end{equation*}
\end{lemma}

\begin{proof}
For $\alpha<\gamma\le(1+H)\alpha$, combining the triangle inequality and \eqref{eq:3.4}, we have
\begin{equation*}
\begin{split}
&\left\|v(t)-v(s)\right\|_{L^2(D,U)}\\
&\lesssim\left\|A^{\frac{\alpha}{2}}\left(\sin\left(A^{\frac{\alpha}{2}}t\right)-\sin\left(A^{\frac{\alpha}{2}}s\right)\right)u_0\right\|_{L^2(D,U)}+\left\|\left(\cos\left(A^{\frac{\alpha}{2}}t\right)-\cos\left(A^{\frac{\alpha}{2}}s\right)\right)v_0\right\|_{L^2(D,U)}\\
&~~~~+\left\|\int^t_0\cos\left(A^{\frac{\alpha}{2}}(t-r)\right)f\left(u(r)\right)\mathrm{d}r-\int^s_0\cos\left(A^{\frac{\alpha}{2}}(s-r)\right)f\left(u(r)\right)\mathrm{d}r\right\|_{L^2(D,U)}\\
&~~~~+\left\|\int^t_0\cos\left(A^{\frac{\alpha}{2}}(t-r)\right)\mathrm{d}B_H(r)-\int^s_0\cos\left(A^{\frac{\alpha}{2}}(s-r)\right)\mathrm{d}B_H(r)\right\|_{L^2(D,U)}\\
&\lesssim(t-s)^{\frac{\gamma-\alpha}{\alpha}}\left\|A^{\frac{\gamma}{2}}u_0\right\|_{L^2(D,U)}+(t-s)^{\frac{\gamma-\alpha}{\alpha}}\left\|A^{\frac{\gamma-\alpha}{2}}v_0\right\|_{L^2(D,U)}+I_1+I_2\\
\end{split}.
\end{equation*}
Using Assumptions \ref{as:2.1} leads to
\begin{equation*}
\begin{split}
I_1\lesssim&\left\|\int^t_s\cos\left(A^{\frac{\alpha}{2}}(t-r)\right)f\left(u(r)\right)\mathrm{d}r\right\|_{L^2(D,U)}\\
&+\left\|\int^s_0\left(\cos\left(A^{\frac{\alpha}{2}}(t-r)\right)-\cos\left(A^{\frac{\alpha}{2}}(s-r)\right)\right)f\left(u(r)\right)\mathrm{d}r\right\|_{L^2(D,U)}\\
\lesssim&(t-s)\left\|A^{\frac{\gamma}{2}}u_0\right\|_{L^2(D,U)}
\end{split}.
\end{equation*}
Based on Eqs. \eqref{eq:3.5-1} and \eqref{eq:3.5-0}, we have
\begin{equation*}
\begin{split}
I_2\lesssim&\left\|\int^t_s\cos\left(A^{\frac{\alpha}{2}}(t-r)\right)\mathrm{d}B_H(r)\right\|_{L^2(D,U)}\\
&+\left\|\int^s_0\left(\cos\left(A^{\frac{\alpha}{2}}(t-r)\right)-\ cos\left(A^{\frac{\alpha}{2}}(s-r)\right)\right)\mathrm{d}B_H(r)\right\|_{L^2(D,U)}\\
\lesssim&(t-s)^{H}+\left\|\int^s_0\left(\cos\left(A^{\frac{\alpha}{2}}(t-r)\right)-\ cos\left(A^{\frac{\alpha}{2}}(s-r)\right)\right)\mathrm{d}B_H(r)\right\|_{L^2(D,U)}\\
\lesssim&(t-s)^{\frac{\gamma-\alpha}{\alpha}}\frac{t^H}{\varepsilon}
\end{split}.
\end{equation*}
Then
\begin{equation*}
\left\|v(t)-v(s)\right\|_{L^2(D,U)}\lesssim (t-s)^{\frac{\gamma-\alpha}{\alpha}}\left(\frac{t^H}{\varepsilon}+\left\|u_0\right\|_{L^2(D,\dot{U}^\gamma)}+\left\|v_0\right\|_{L^2(D,\dot{U}^{\gamma-\alpha})}\right).
\end{equation*}

Similarly, we can obtain the remaining estimates by using the above method.
\end{proof}

For $\gamma\le\alpha$, Lemma \ref{le:1} implies that
\begin{equation*}
\lim_{s\to t}\left\|v(t)-v(s)\right\|_{L^2(D,U)}\neq0
\end{equation*}
Thus, the error can not converge to $0$ under mean-squared $L^2$-norm if using the semi-implicit Euler discretize the following equation in time
\begin{equation*}
u(t_{n+1})-u(t_{n})=\int^{t_{n+1}}_{t_{n}}v(s)\mathrm{d}s.
\end{equation*}
Then, an interesting question arises as to whether it is possible to find an effective difference method to obtain the time approximation of Eq. (3.2) in this case. In this paper, we provide a positive answer. Our method is to give an equivalent form of Eq. (3.2), the H\"older regularity of whose solution is better than the one of the solution of Eq. \eqref{eq:3.2}. In fact, combining Ornstein-Uhlenbeck process and Eq. \eqref{eq:3.2}, one can get the anticipated quivalent form \eqref{eq:3.7}. Let
\begin{equation}\label{eq:3.6-1}
Z(t)=X(t)-\int^t_0\mathrm{e}^{\Lambda(t-s)}\left[\begin{array}{c}
0\\
I
\end{array}\right]\mathrm{d}B_H(s),
\end{equation}
where
\begin{equation*}
Z(t)=
\left[\begin{array}{c}
z(t)\\[1.5mm]
\dot{z}(t)
\end{array} \right].
\end{equation*}

If $X(t)$ is the unique mild solution of \eqref{eq:3.2}, then $Z(t)$ is the unique mild solution of the  partial differential equation
\begin{equation}\label{eq:3.7}
\frac{\mathrm{d}}{\mathrm{d}t}Z(t)=\Lambda Z(t)+\left[\begin{array}{c}
0\\
f(u(t))
\end{array} \right] ~ {\rm for } ~ t\in(0,T] ~ {\rm with } ~ Z(0)=X(0),
\end{equation}
where
\begin{equation*}
f\left(u(t)\right)=f\left(z(t)+\int^{t}_0A^{-\frac{\alpha}{2}}\sin\left(A^{\frac{\alpha}{2}}(t-s)\right)\mathrm{d}B_H(s)\right).
\end{equation*}
The unique mild solution of \eqref{eq:3.7} is given by
\begin{equation}\label{eq:3.7-1}
Z(t)=\mathrm{e}^{\Lambda t}Z(0)+\int^t_0\mathrm{e}^{\Lambda (t-s)}\left[\begin{array}{c}
0\\
f(u(t))
\end{array} \right]\mathrm{d}s.
\end{equation}
Then we can obtain $z(t)$ and $\dot{z}(t)$ as
\begin{equation}\label{eq:3.7-2}
\begin{split}
z(t)&=\cos\left(A^{\frac{\alpha}{2}}t\right)u_0+A^{-\frac{\alpha}{2}}\sin\left(A^{\frac{\alpha}{2}}t\right)v_0+\int^t_0A^{-\frac{\alpha}{2}}\sin\left(A^{\frac{\alpha}{2}}(t-s)\right)f\left(u(s)\right)\mathrm{d}s,\\
\dot{z}(t)&=-A^{\frac{\alpha}{2}}\sin\left(A^{\frac{\alpha}{2}}t\right)u_0+\cos\left(A^{\frac{\alpha}{2}}t\right)v_0+\int^t_0\cos\left(A^{\frac{\alpha}{2}}(t-s)\right)f\left(u(s)\right)\mathrm{d}s.
\end{split}
\end{equation}
Then, combining Eqs. \eqref{eq:3.4} and \eqref{eq:3.7-2} leads to
\begin{equation}\label{eq:3.7-3}
\left\{
\begin{array}{ll}
u(t)=z(t)+\int^{t}_0A^{-\frac{\alpha}{2}}\sin\left(A^{\frac{\alpha}{2}}(t-s)\right)\mathrm{d}B_H(s)\\
\dot{u}(t)=\dot{z}(t)+\int^{t}_0\cos\left(A^{\frac{\alpha}{2}}(t-s)\right)\mathrm{d}B_H(s)
\end{array}
\right.,
\end{equation}
which agree with Eq. \eqref{eq:3.6-1}, as was expected.

Equations \eqref{eq:3.7} and \eqref{eq:3.7-3} will be used to obtain the temporal semi-discretization of \eqref{eq:3.2}. Therefore we give the following regularity estimates, which will be used to design numerical schemes and discuss the convergence behavior of errors.

\begin{proposition}\label{pro:3}
Suppose that Assumptions \ref{as:2.1} are satisfied, $\varepsilon>0$, $\gamma=\alpha+2\rho-\frac{d+\varepsilon}{2}$, $0<\gamma\le\alpha$, $\left\|u_0\right\|_{L^p(D,\dot{U}^{\gamma+\alpha})}<\infty$, and $\left\|v_0\right\|_{L^p(D,\dot{U}^{\gamma})}<\infty$. Then
\begin{equation*}
\begin{split}
&\left\|A^{\frac{\alpha}{2}}\left(z(t)-z(s)\right)\right\|_{L^2(D,U)}+\left\|\dot{z}(t)-\dot{z}(s)\right\|_{L^2(D,U)}\\
&\lesssim (t-s)^{\frac{\gamma}{\alpha}}\left(\sqrt{\frac{1}{\varepsilon}}+\left\|u_0\right\|_{L^2(D,\dot{U}^{\gamma+\alpha})}+\left\|v_0\right\|_{L^2(D,\dot{U}^{\gamma})}\right).
\end{split}
\end{equation*}
\end{proposition}

\begin{proof}
Firstly, we give two inequalities
\begin{equation*}
|\sin(a)-\sin(b)|\lesssim|\sin(a)-\sin(b)|^\theta=\left|\int^a_b\cos(t)\mathrm{d}t\right|^\theta\lesssim|a-b|^\theta
\end{equation*}
and
\begin{equation*}
|\cos(a)-\cos(b)|\lesssim|\cos(a)-\cos(b)|^\theta=\left|-\int^a_b\sin(t)\mathrm{d}t\right|^\theta\lesssim|a-b|^\theta,
\end{equation*}
where $0\le \theta\le1$. Then, combining Eq. \eqref{eq:3.7-2}, triangle inequality and Lemma \ref{le:1}, the H\"older regularity of $A^{\frac{\alpha}{2}}z(t)$ and $\dot{z}(t)$ can be established as
\begin{equation*}
\begin{split}
&\left\|A^{\frac{\alpha}{2}}\left(z(t)-z(s)\right)\right\|_{L^2(D,U)}\\
&\lesssim\left\|A^{\frac{\alpha}{2}}\left[\cos\left(A^{\frac{\alpha}{2}}t\right)-\cos\left(A^{\frac{\alpha}{2}}s\right)\right]u_0\right\|_{L^2(D,U)}\\
&~~~+\left\|\left[\sin\left(A^{\frac{\alpha}{2}}t\right)-\sin\left(A^{\frac{\alpha}{2}}s\right)\right]v_0\right\|_{L^2(D,U)}\\
&~~~+\int^t_s\left\|\sin\left(A^{\frac{\alpha}{2}}(t-r)\right)f\left(u(r)\right)\right\|_{L^2(D,U)}\mathrm{d}r\\
&~~~+\int^s_0\left\|\left[\sin\left(A^{\frac{\alpha}{2}}(t-r)\right)-\sin\left(A^{\frac{\alpha}{2}}(s-r)\right)\right]f\left(u(r)\right)\right\|_{L^2(D,U)}\mathrm{d}r\\
&\lesssim(t-s)^{\frac{\gamma}{\alpha}}\left(1+\left\|u_0\right\|_{L^2(D,\dot{U}^{\gamma+\alpha})}+\left\|v_0\right\|_{L^2(D,\dot{U}^{\gamma})}\right)
\end{split}
\end{equation*}
and
\begin{equation*}
\begin{split}
&\left\|\dot{z}(t)-\dot{z}(s)\right\|_{L^2(D,U)}\\
&\lesssim\left\|A^{\frac{\alpha}{2}}\left(\sin\left(A^{\frac{\alpha}{2}}t\right)-\sin\left(A^{\frac{\alpha}{2}}s\right)\right)u_0\right\|_{L^2(D,U)}\\
&~~~~+\left\|\left(\cos\left(A^{\frac{\alpha}{2}}t\right)-\cos\left(A^{\frac{\alpha}{2}}s\right)\right)v_0\right\|_{L^2(D,U)}\\
&~~~~+\left\|\int^t_s\cos\left(A^{\frac{\alpha}{2}}(t-r)\right)f\left(u(r)\right)\mathrm{d}r\right\|_{L^2(D,U)}\\
&~~~~+\left\|\int^s_0\cos\left(A^{\frac{\alpha}{2}}(t-r)\right)-\cos\left(A^{\frac{\alpha}{2}}(s-r)\right)f\left(u(r)\right)\mathrm{d}r\right\|_{L^2(D,U)}\\
&\lesssim(t-s)^{\frac{\gamma}{\alpha}}\left\|A^{\frac{\alpha+\gamma}{2}}u_0\right\|_{L^2(D,U)}+(t-s)^{\frac{\gamma}{\alpha}}\left\|A^{\frac{\gamma}{2}}v_0\right\|_{L^2(D,U)}\\
&~~~~+\int^t_s\left\|f\left(u(r)\right)\right\|_{L^2(D,U)}\mathrm{d}r+(t-s)^{\frac{\gamma}{\alpha}}\int^s_0\left\|A^{\frac{\gamma}{2}}f\left(u(r)\right)\right\|_{L^2(D,U)}\mathrm{d}r\\
&\lesssim (t-s)^{\frac{\gamma}{\alpha}}\left(\sqrt{\frac{1}{\varepsilon}}+\left\|u_0\right\|_{L^2(D,\dot{U}^{\gamma+\alpha})}+\left\|v_0\right\|_{L^2(D,\dot{U}^{\gamma})}\right)
\end{split}.
\end{equation*}
\end{proof}

To discuss the numerical analysis of the higher order time discretization, we establish the H\"older regularity estimates of $A^{\frac{\alpha}{2}}\dot{z}(t)$ and $\ddot{z}(t)$. Using Eqs. \eqref{eq:3.7} and \eqref{eq:3.7-2}, we can get Proposition \ref{pro:4}.
\begin{proposition}\label{pro:4}
Suppose that Assumption \ref{as:2.1} is satisfied, $\left\|u_0\right\|_{L^p(D,\dot{U}^{\gamma+\alpha})}<\infty$, $\left\|v_0\right\|_{L^p(D,\dot{U}^{\gamma})}<\infty$, $\varepsilon>0$, $\gamma=\alpha+2\rho-\frac{d+\varepsilon}{2}$, and $\gamma>\alpha$.
Then
\begin{equation*}
\begin{split}
&\left\|A^{\frac{\alpha}{2}}\left(\dot{z}(t)-\dot{z}(s)\right)\right\|_{L^2(D,U)}+\left\|\ddot{z}(t)-\ddot{z}(s)\right\|_{L^2(D,U)}\\
&\lesssim (t-s)^{\min\{\frac{\gamma-\alpha}{\alpha},1\}}\left(1+\left\|u_0\right\|_{L^2(D,\dot{U}^{\gamma+\alpha})}+\left\|v_0\right\|_{L^2(D,\dot{U}^{\gamma})}\right)
\end{split}.
\end{equation*}
\end{proposition}
\begin{proof}
 A brief derivation is given as
\begin{equation*}
\begin{split}
\left\|\ddot{z}(t)-\ddot{z}(s)\right\|_{L^2(D,U)}&\lesssim \left\|A^\alpha(z(t)-z(s))\right\|_{L^2(D,U)}+\left\|f\left(u(t)\right)-f\left(u(s)\right)\right\|_{L^2(D,U)}\\
&\lesssim(t-s)^{\min\{\frac{\gamma-\alpha}{\alpha},1\}}\left(1+\left\|u_0\right\|_{L^2(D,\dot{U}^{\gamma+\alpha})}+\left\|v_0\right\|_{L^2(D,\dot{U}^{\gamma})}\right)
\end{split}.
\end{equation*}
\end{proof}

\section{Temporal Discretization} \label{sec:4}
In this section, we concern the time discretization of \eqref{eq:3.2}. Based on Propositions \ref{pro:3} and \ref{pro:4}, we design two discrete schemes. Meanwhile stability and error estimates of the discrete schemes are derived.

\subsection{Low order time discretization}
Using Eq. \eqref{eq:3.7}, we have
\begin{equation}\label{eq:4.0-0}
\left\{
\begin{array}{ll}
\dot{z}(t_{n+1})-\dot{z}(t_{n})=-\int^{t_{n+1}}_ {t_{n}}A^\alpha z(s)\mathrm{d}s+\int^{t_{n+1}}_ {t_{n}}f(u(s))\mathrm{d}s\\
z(t_{n+1})-z(t_{n})=\int^{t_{n+1}}_ {t_{n}}\dot{z}(s)\mathrm{d}s
\end{array}
\right..
\end{equation}

The convergence rate of difference methods is generally limited by the smoothness of the solution. As $0<\gamma\le\alpha$, Proposition \ref{pro:3} shows $A^{\frac{\alpha}{2}}z(t)$ and $\dot{z}(t)$ are H\"older continuous in the sense of mean-squared $L^p$-norm. Thus, we apply rectangle formula to discretize the integrals in Eq. \eqref{eq:4.0-0}. Let $z_{n}$ and $\dot{z}_{n}$ denote respectively the approximation of $z(t_{n})$ and $\dot{z}(t_{n})$ with
fixed time step size $\tau=\frac{T}{N}$ and $t_n=n\tau$ $(n=0,1,2,\dots,N)$. Then, we can get temporal discretization of Eq. \eqref{eq:4.0-0}
\begin{equation}\label{eq:4.0-1}
\left\{
\begin{array}{ll}
\dot{z}_{n+1}-\dot{z}_{n}=-\tau A^\alpha z_{n+1}+\tau f(u_{n})\\
z_{n+1}-z_{n}=\tau \dot{z}_{n+1}
\end{array}
\right.,
\end{equation}
where $u_{n}$ is the numerical solution of Eq. \eqref{eq:3.1} in time. We give the approximation of $\int^{t_{n+1}}_0A^{-\frac{\alpha}{2}}\sin\left(A^{\frac{\alpha}{2}}(t_{n+1}-s)\right)\mathrm{d}B_H(s)$ to obtain $u_{n+1}$, that is
\begin{equation}\label{eq:4.0-0-0}
\sum^{n}_{i=0}\int^{t_{i+1}}_{t_{i}}A^{-\frac{\alpha}{2}}\sin\left(A^{\frac{\alpha}{2}}(t_{n+1}-t_{i})\right)\mathrm{d}B_H(s).
\end{equation}
Then, using Eq. \eqref{eq:3.7-3}, we have
\begin{equation*}
u_{0}=z_{0}
\end{equation*}
and
\begin{equation}\label{eq:4.0-0-1}
u_{n+1}=z_{n+1}+\sum^{n}_{i=0}\int^{t_{i+1}}_{t_{i}}A^{-\frac{\alpha}{2}}\sin\left(A^{\frac{\alpha}{2}}(t_{n+1}-t_{i})\right)\mathrm{d}B_H(s).
\end{equation}

\begin{theorem}\label{th:1}
Let $\dot{z}_{n+1}$ and $z_{n+1}$ be expressed by Eq. \eqref{eq:4.0-1}. Suppose that Assumptions \ref{as:2.1} are satisfied, $\left\|u_0\right\|_{L^p(D,\dot{U}^{\gamma+\alpha})}<\infty$, $\left\|v_0\right\|_{L^p(D,\dot{U}^{\gamma})}<\infty$, $\varepsilon>0$, $\gamma=\alpha+2\rho-\frac{d+\varepsilon}{2}$, and $0<\gamma\le\alpha$. Then
\begin{equation*}
\left\|u_{n+1}\right\|_{L^2(D,U)}\lesssim\left\|A^{-\frac{\alpha}{2}}v_{0}\right\|_{L^2(D,U)}+\left\|u_{0}\right\|_{L^2(D,U)}+1.
\end{equation*}
\end{theorem}
\begin{proof}
By using the weak formulations of Eq. \eqref{eq:4.0-1}, we have
\begin{equation*}
\left\langle\dot{z}_{n+1}-\dot{z}_{n},A^{-\alpha}\dot{z}_{n+1}\right\rangle=\left\langle-\tau A^{\alpha} z_{n+1}+\tau f\left(u_{n}\right),A^{-\alpha}\dot{z}_{n+1}\right\rangle
\end{equation*}
and
\begin{equation*}
\left\langle z_{n+1}-z_{n},z_{n+1}\right\rangle=\left\langle\tau \dot{z}_{n+1},z_{n+1}\right\rangle.
\end{equation*}
Using the fact $(a-b)a=\frac{1}{2}(a^2-b^2)+\frac{1}{2}(a-b)^2$, we get
\begin{equation*}
\begin{split}
&\mathrm{E}\left[\left\|A^{-\frac{\alpha}{2}}\dot{z}_{n+1}\right\|^2\right]-\mathrm{E}\left[\left\|A^{-\frac{\alpha}{2}}\dot{z}_{n}\right\|^2\right]+\mathrm{E}\left[\left\|z_{n+1}\right\|^2\right]-\mathrm{E}\left[\left\|z_{n}\right\|^2\right]\\
&\lesssim\tau\left(\mathrm{E}\left[\left\|f\left(u_{n}\right)\right\|^2\right]+\mathrm{E}\left[\left\|A^{-\frac{\alpha}{2}}\dot{z}_{n+1}\right\|^2\right]\right)
\end{split}.
\end{equation*}
Then
\begin{equation*}
\begin{split}
&\mathrm{E}\left[\left\|A^{-\frac{\alpha}{2}}\dot{z}_{n+1}\right\|^2\right]+\mathrm{E}\left[\left\|z_{n+1}\right\|^2\right]\\
&\lesssim\tau\sum^{n+1}_{i=1}\left(\mathrm{E}\left[\left\|f\left(u_{i}\right)\right\|^2\right]+\mathrm{E}\left[\left\|A^{-\frac{\alpha}{2}}\dot{z}_{i}\right\|^2\right]\right)+\mathrm{E}\left[\left\|A^{-\frac{\alpha}{2}}\dot{z}_{0}\right\|^2\right]+\mathrm{E}\left[\left\|z_{0}\right\|^2\right]\\
&\lesssim\tau\sum^{n+1}_{i=1}\left(\mathrm{E}\left[\left\|u_{i}\right\|^2\right]+\mathrm{E}\left[\left\|A^{-\frac{\alpha}{2}}\dot{z}_{i}\right\|^2\right]\right)+\mathrm{E}\left[\left\|A^{-\frac{\alpha}{2}}\dot{z}_{0}\right\|^2\right]+\mathrm{E}\left[\left\|z_{0}\right\|^2\right]+1
\end{split}.
\end{equation*}
Then, using Eq. \eqref{eq:4.0-0-1} and the discrete Gr\"onwall inequality leads to
\begin{equation*}
\mathrm{E}\left[\left\|A^{-\frac{\alpha}{2}}\dot{z}_{n+1}\right\|^2\right]+\mathrm{E}\left[\left\|z_{n+1}\right\|^2\right]\lesssim\mathrm{E}\left[\left\|A^{-\frac{\alpha}{2}}\dot{z}_{0}\right\|^2\right]+\mathrm{E}\left[\left\|z_{0}\right\|^2\right]+1,
\end{equation*}
which implies
\begin{equation*}
\begin{split}
\mathrm{E}\left[\left\|u_{n+1}\right\|^2\right]&\lesssim\mathrm{E}\left[\left\|z_{n+1}\right\|^2\right]+\mathrm{E}\left[\left\|\sum^{n}_{i=0}\int^{t_{i+1}}_{t_{i}}A^{-\frac{\alpha}{2}}\sin\left(A^{\frac{\alpha}{2}}(t_{n+1}-t_{i})\right)\mathrm{d}B_H(s)\right\|^2\right]\\
&\lesssim\mathrm{E}\left[\left\|A^{-\frac{\alpha}{2}}\dot{z}_{0}\right\|^2\right]+\mathrm{E}\left[\left\|z_{0}\right\|^2\right]+1
\end{split}.
\end{equation*}
\end{proof}

Next, we discuss the error behavior. Combining Lemma \ref{le:1}, Proposition \ref{pro:3}, Eqs. \eqref{eq:3.7-3} and \eqref{eq:4.0-0-1}, we get the following theorem.

\begin{theorem}\label{th:2}
Let $\dot{z}_{n+1}$ and $z_{n+1}$ be expressed by Eq. \eqref{eq:4.0-1}. Suppose that Assumptions \ref{as:2.1} are satisfied, $\left\|u_0\right\|_{L^p(D,\dot{U}^{\gamma+\alpha})}<\infty$, $\left\|v_0\right\|_{L^p(D,\dot{U}^{\gamma})}<\infty$, $\varepsilon>0$, $\gamma=\alpha+2\rho-\frac{d+\varepsilon}{2}$, and $0<\gamma\le\alpha$. Then
\begin{equation*}
\left\|u(t_{n+1})-u_{n+1}\right\|_{L^2(D,U)}\lesssim\tau^{\frac{\gamma}{\alpha}}\left(\sqrt{\frac{1}{\varepsilon}}+\left\|u_0\right\|_{L^2(D,\dot{U}^{\gamma+\alpha})}+\left\|v_0\right\|_{L^2(D,\dot{U}^{\gamma})}\right).
\end{equation*}
\end{theorem}
\begin{proof}
Using Eqs. \eqref{eq:3.7-3} and \eqref{eq:4.0-0-1} leads to
\begin{equation*}
\begin{split}
u(t_{n+1})-u_{n+1}=&z(t_{n+1})-z_{n+1}+\int^{t_{n+1}}_0A^{-\frac{\alpha}{2}}\sin\left(A^{\frac{\alpha}{2}}(t_{n+1}-s)\right)\mathrm{d}B_H(s)\\
&-\sum^{n}_{i=0}\int^{t_{i+1}}_{t_{i}}A^{-\frac{\alpha}{2}}\sin\left(A^{\frac{\alpha}{2}}(t_{n+1}-t_{i})\right)\mathrm{d}B_H(s)
\end{split}.
\end{equation*}
Then
\begin{equation}\label{eq:4.0-1-2}
\begin{split}
\left\|u(t_{n+1})-u_{n+1}\right\|_{L^2(D,U)}&\le\left\|z(t_{n+1})-z_{n+1}\right\|_{L^2(D,U)}\\
&~~~~+\left\|\int^{t_{n+1}}_0A^{-\frac{\alpha}{2}}\sin\left(A^{\frac{\alpha}{2}}(t_{n+1}-s)\right)\mathrm{d}B_H(s)\right.\\
&~~~~\left.-\sum^{n}_{i=0}\int^{t_{i+1}}_{t_{i}}A^{-\frac{\alpha}{2}}\sin\left(A^{\frac{\alpha}{2}}(t_{n+1}-t_{i})\right)\mathrm{d}B_H(s)\right\|_{L^2(D,U)}
\end{split}.
\end{equation}
Lemma \ref{le:02} implies
\begin{equation}\label{eq:4.0-0-2}
\begin{split}
&\mathrm{E}\left[\left\|\int^{t_{n+1}}_0A^{-\frac{\alpha}{2}}\sin\left(A^{\frac{\alpha}{2}}(t_{n+1}-s)\right)\mathrm{d}B_H(s)-\sum^{n}_{i=0}\int^{t_{i+1}}_{t_{i}}A^{-\frac{\alpha}{2}}\sin\left(A^{\frac{\alpha}{2}}(t_{n+1}-t_{i})\right)\mathrm{d}B_H(s)\right\|^2\right]\\
&\lesssim\frac{1}{\varepsilon}\tau^{\frac{2\gamma}{\alpha}}
\end{split}.
\end{equation}
Let
\begin{equation*}
\left\{
\begin{array}{ll}
\dot{e}_{n+1}=\dot{z}(t_{n+1})-\dot{z}_{n+1}\\
e_{n+1}=z(t_{n+1})-z_{n+1}
\end{array}
\right..
\end{equation*}
Using Eqs. \eqref{eq:4.0-0} and \eqref{eq:4.0-1} leads to
\begin{equation*}
\left\{
\begin{array}{ll}
\dot{e}_{n+1}-\dot{e}_{n}=-\int^{t_{n+1}}_{t_{n}}A^\alpha\left(z(s)-z_{n+1}\right)\mathrm{d}s+\int^{t_{n+1}}_{t_{n}}\left(f(u(s))-f\left(u_{n}\right)\right)\mathrm{d}s\\
e_{n+1}-e_{n}=\int^{t_{n+1}}_{t_{n}}\left(\dot{z}(s)-\dot{z}_{n+1}\right)\mathrm{d}s
\end{array}
\right..
\end{equation*}
Then, H\"older inequality implies
\begin{equation}\label{eq:4.0-2}
\begin{split}
&\mathrm{E}\left[\left\langle\dot{e}_{n+1}-\dot{e}_{n},A^{-\alpha}\dot{e}_{n+1}\right\rangle\right]\\
&=-\mathrm{E}\left[\left\langle\int^{t_{n+1}}_{t_{n}}A^\alpha\left(z(s)-z_{n+1}\right)\mathrm{d}s,A^{-\alpha}\dot{e}_{n+1}\right\rangle\right]\\
&~~~~+\mathrm{E}\left[\left\langle\int^{t_{n+1}}_{t_{n}}\left(f(u(s))-f\left(u_{n}\right)\right)\mathrm{d}s,A^{-\alpha}\dot{e}_{n+1}\right\rangle\right]\\
&\lesssim\mathrm{E}\left[\int^{t_{n+1}}_{t_{n}}\left\|A^{\frac{\alpha}{2}}\left(z(s)-z(t_{n+1})\right)\right\|\left\|A^{-\frac{\alpha}{2}}\dot{e}_{n+1}\right\|\mathrm{d}s\right]\\
&~~~~-\mathrm{E}\left[\int^{t_{n+1}}_{t_{n}}\left\langle e_{n+1},\dot{e}_{n+1}\right\rangle\mathrm{d}s\right]+\mathrm{E}\left[\int^{t_{n+1}}_{t_{n}}\left\|\left(f(u(s))-f\left(u_{n}\right)\right)\right\|\left\|A^{-\frac{\alpha}{2}}\dot{e}_{n+1}\right\|\mathrm{d}s\right]
\end{split}
\end{equation}
and
\begin{equation}\label{eq:4.0-3}
\begin{split}
\mathrm{E}\left[\left\langle e_{n+1}-e_{n},e_{n+1}\right\rangle\right]=&\mathrm{E}\left[\int^{t_{n+1}}_{t_{n}}\left\langle \dot{z}(s)-\dot{z}_{n+1},e_{n+1}\right\rangle\mathrm{d}s\right]\\
\lesssim&\mathrm{E}\left[\int^{t_{n+1}}_{t_{n}}\left\| \dot{z}(s)-\dot{z}(t_{n+1})\right\|\left\|e_{n+1}\right\|\mathrm{d}s\right]\\
&+\mathrm{E}\left[\int^{t_{n+1}}_{t_{n}}\left\langle \dot{e}_{n+1},e_{n+1}\right\rangle\mathrm{d}s\right]\\
\end{split}
\end{equation}
Using Lemma\ref{le:1}, Proposition \ref{pro:3}, Eqs. \eqref{eq:4.0-0-2}, \eqref{eq:4.0-2} and \eqref{eq:4.0-3}, we get
\begin{equation*}
\begin{split}
&\mathrm{E}\left[\left\|A^{-\frac{\alpha}{2}}\dot{e}_{n+1}\right\|^2-\left\|A^{-\frac{\alpha}{2}}\dot{e}_{n}\right\|^2\right]+\mathrm{E}\left[\left\|e_{n+1}\right\|^2-\left\|e_{n}\right\|^2\right]\\
&\lesssim\mathrm{E}\left[\int^{t_{n+1}}_{t_{n}}\left\|A^{\frac{\alpha}{2}}\left(z(s)-z(t_{n+1})\right)\right\|\left\|A^{-\frac{\alpha}{2}}\dot{e}_{n+1}\right\|\mathrm{d}s\right]\\
&~~~~+\mathrm{E}\left[\int^{t_{n+1}}_{t_{n}}\left\|\left(f(u(s))-f\left(u_{n}\right)\right)\right\|\left\|A^{-\frac{\alpha}{2}}\dot{e}_{n+1}\right\|\mathrm{d}s\right]\\
&~~~~+\mathrm{E}\left[\int^{t_{n+1}}_{t_{n}}\left\|
 \dot{z}(s)-\dot{z}(t_{n+1})\right\|\left\|e_{n+1}\right\|\mathrm{d}s\right]\\
&\lesssim\mathrm{E}\left[\int^{t_{n+1}}_{t_{n}}\left\|A^{\frac{\alpha}{2}}\left(z(s)-z(t_{n+1})\right)\right\|^2\mathrm{d}s\right]+\tau\mathrm{E}\left[\left\|A^{-\frac{\alpha}{2}}\dot{e}_{n+1}\right\|^2\right]\\
&~~~~+\mathrm{E}\left[\int^{t_{n+1}}_{t_{n}}\left\|u(s)-u(t_{n})\right\|^2\mathrm{d}s\right]+\tau\mathrm{E}\left[\left\|u(t_{n})-u_{n}\right\|^2+\left\|A^{-\frac{\alpha}{2}}\dot{e}_{n+1}\right\|^2\right]\\
&~~~~+\mathrm{E}\left[\int^{t_{n+1}}_{t_{n}}\left\|
\dot{z}(s)-\dot{z}(t_{n+1})\right\|^2\mathrm{d}s\right]+\tau\mathrm{E}\left[\left\|e_{n+1}\right\|^2\right]\\
&\lesssim\tau^{1+\frac{2\gamma}{\alpha}}\left(\frac{1}{\varepsilon}+\left\|u_0\right\|^2_{L^2(D,\dot{U}^{\gamma+\alpha})}+\left\|v_0\right\|^2_{L^2(D,\dot{U}^{\gamma})}\right)\\
&~~~~+\tau\mathrm{E}\left[\left\|A^{-\frac{\alpha}{2}}\dot{e}_{n+1}\right\|^2\right]+\tau\mathrm{E}\left[\left\|e_{n+1}\right\|^2\right]\\
\end{split}.
\end{equation*}
Then
\begin{equation}\label{eq:4.0-4}
\begin{split}
&\mathrm{E}\left[\left\|A^{-\frac{\alpha}{2}}\dot{e}_{n+1}\right\|^2\right]+\mathrm{E}\left[\left\|e_{n+1}\right\|^2\right]\\
&\lesssim\tau^{\frac{2\gamma}{\alpha}}\left(\frac{1}{\varepsilon}+\left\|u_0\right\|^2_{L^2(D,\dot{U}^{\gamma+\alpha})}+\left\|v_0\right\|^2_{L^2(D,\dot{U}^{\gamma})}\right)\\
&~~~~+\tau\sum^{n+1}_{i=1}\mathrm{E}\left[\left\|A^{-\frac{\alpha}{2}}\dot{e}_{i}\right\|^2\right]
+\tau\sum^{n+1}_{i=1}\mathrm{E}\left[\left\|e_{i}\right\|^2\right]\\
\end{split}.
\end{equation}
Combining the discrete Gr\"onwall inequality, Eqs. \eqref{eq:4.0-0-2} and \eqref{eq:4.0-4} leads to
\begin{equation*}
\begin{split}
\mathrm{E}\left[\left\|u(t_{n+1})-u_{n+1}\right\|^2\right]&\lesssim\mathrm{E}\left[\left\|e_{n+1}\right\|^2\right]+\frac{1}{\varepsilon}\tau^{\frac{2\gamma}{\alpha}}\\
&\lesssim\tau^{\frac{2\gamma}{\alpha}}\left(\frac{1}{\varepsilon}+\left\|u_0\right\|^2_{L^2(D,\dot{U}^{\gamma+\alpha})}+\left\|v_0\right\|^2_{L^2(D,\dot{U}^{\gamma})}\right)
\end{split}.
\end{equation*}
\end{proof}

\subsection{Higher order time discretization}
As $\gamma>\alpha$, $v(t)$ is H\"older continuous. Then we can design a higher order scheme that has the superlinear convergence rate in time. In this case, the time regularity estimates of $A^{\frac{\alpha}{2}}\dot{z}(t)$ and $\ddot{z}(t)$ imply that using trapezoidal rule, we can get the high accuracy approximations of the integrals in Eq. \eqref{eq:4.0-0}, i.e.,
\begin{equation*}
\int^{t_{n+1}}_ {t_{n}}A^\alpha z(s)\mathrm{d}s\approx \int^{t_{n+1}}_ {t_{n}} A^\alpha \frac{z_{n+1}+z_{n}}{2}\mathrm{d}s
\end{equation*}
and
\begin{equation*}
\int^{t_{n+1}}_ {t_{n}}\dot{z}(s)\mathrm{d}s\approx\int^{t_{n+1}}_ {t_{n}} \frac{\dot{z}_{n+1}+\dot{z}_{n}}{2}\mathrm{d}s.
\end{equation*}
Then, we can design a higher order time discretization of Eq. \eqref{eq:3.7}. Modifying the scheme \eqref{eq:4.0-1} leads to the following equation
\begin{equation}\label{eq:4.0-5}
\left\{
\begin{array}{ll}
\dot{z}_{n+1}-\dot{z}_{n}=-\tau A^\alpha \frac{z_{n+1}+z_{n}}{2}+\tau f(u_{n})+\frac{\tau}{2}\left(f(u_{n})-f(u_{n-1})\right)\\
z_{n+1}-z_{n}=\tau \frac{\dot{z}_{n+1}+\dot{z}_{n}}{2}
\end{array}
\right..
\end{equation}
To obtain the desired convergence rate, we improve the accuracy of approximation for stochastic integral by the following scheme
\begin{equation}\label{eq:stochastic}
\sum^{n}_{i=0}\int^{t_{i+1}}_{t_{i}}\left(A^{-\frac{\alpha}{2}}\sin\left(A^{\frac{\alpha}{2}}(t_{n+1}-t_{i})\right)-(s-t_{i})\cos\left(A^{\frac{\alpha}{2}}(t_{n+1}-t_{i})\right)\right)\mathrm{d}B_H(s).
\end{equation}
Then
\begin{equation}\label{eq:4.0-0-3}
u_{n+1}=z_{n+1}+\sum^{n}_{i=0}\int^{t_{i+1}}_{t_{i}}\left(A^{-\frac{\alpha}{2}}\sin\left(A^{\frac{\alpha}{2}}(t_{n+1}-t_{i})\right)-(s-t_{i})\cos\left(A^{\frac{\alpha}{2}}(t_{n+1}-t_{i})\right)\right)\mathrm{d}B_H(s).
\end{equation}
In \cite{LiuDeng}, the error estimate of the approximation for stochastic
integral is given as
\begin{equation}\label{eq:4.0-0-4}
\begin{split}
&\mathrm{E}\left[\left\|\sum^{n}_{i=0}\int^{t_{i+1}}_{t_{i}}\left(-\int^{s}_{t_{i}}\cos\left(A^{\frac{\alpha}{2}}(t_{n+1}-r)\right)\mathrm{d}r+(s-t_{i})\cos\left(A^{\frac{\alpha}{2}}(t_{n+1}-t_{i})\right)\right)\mathrm{d}B_H(s)\right\|^2\right]\\
&\lesssim\tau^{\min\{\frac{2\gamma}{\alpha},4\}}\sum^\infty_{j=1}\lambda^{\min\{\gamma-\alpha,\alpha\}-2\rho}_j
\end{split}.
\end{equation}
\begin{theorem}\label{th:3}
Let $\dot{z}_{n+1}$ and $z_{n+1}$ be expressed by Eq. \eqref{eq:4.0-5}. Suppose that Assumptions \ref{as:2.1} are satisfied, $\left\|u_0\right\|_{L^p(D,\dot{U}^{\gamma+\alpha})}<\infty$, $\left\|v_0\right\|_{L^p(D,\dot{U}^{\gamma})}<\infty$, $\varepsilon>0$, $\gamma=\alpha+2\rho-\frac{d+\varepsilon}{2}$, and $\gamma>\alpha$. Then
\begin{equation*}
\left\|u_{n+1}\right\|_{L^2(D,U)}\lesssim\left\|A^{-\frac{\alpha}{2}}v_{0}\right\|_{L^2(D,U)}+\left\|u_{0}\right\|_{L^2(D,U)}+1.
\end{equation*}
\end{theorem}
\begin{proof}
Combining weak formulations of Eq. \eqref{eq:4.0-1} and H\"older inequality, we have
\begin{equation}\label{eq:4.0-6}
\begin{split}
\mathrm{E}\left[\left\langle z_{n+1}-z_{n},z_{n+1}+z_{n}\right\rangle\right]=\mathrm{E}\left[\left\langle \tau\frac{\dot{z}_{n+1}+\dot{z}_{n}}{2},z_{n+1}+z_{n}\right\rangle\right]
\end{split}
\end{equation}
and
\begin{equation}\label{eq:4.0-7}
\begin{split}
&\mathrm{E}\left[\left\langle \dot{z}_{n+1}-\dot{z}_{n},A^{-\alpha}\left(\dot{z}_{n+1}+\dot{z}_{n}\right)\right\rangle\right]\\
&=\mathrm{E}\left[\left\langle -\tau A^\alpha \frac{z_{n+1}+z_{n}}{2},A^{-\alpha}\left(\dot{z}_{n+1}+\dot{z}_{n}\right)\right\rangle\right]\\
&~~~~+\mathrm{E}\left[\left\langle \tau f(u_{n}),A^{-\alpha}\left(\dot{z}_{n+1}+\dot{z}_{n}\right)\right\rangle\right]\\
&~~~~+\mathrm{E}\left[\left\langle\frac{\tau}{2}\left(f(u_{n})-f(u_{n-1})\right),A^{-\alpha}\left(\dot{z}_{n+1}+\dot{z}_{n}\right)\right\rangle\right]\\
&\lesssim\mathrm{E}\left[\left\langle -\tau A^\alpha \frac{z_{n+1}+z_{n}}{2},A^{-\alpha}\left(\dot{z}_{n+1}+\dot{z}_{n}\right)\right\rangle\right]\\
&~~~~+\mathrm{E}\left[\tau \left\|f(u_{n})\right\|\left\|A^{-\alpha}\left(\dot{z}_{n+1}+\dot{z}_{n}\right)\right\|\right]\\
&~~~~+\mathrm{E}\left[\frac{\tau}{2}\left\|\left(f(u_{n})-f(u_{n-1})\right)\right\|\left\|A^{-\alpha}\left(\dot{z}_{n+1}+\dot{z}_{n}\right)\right\|\right]\\
&\lesssim\mathrm{E}\left[\left\langle -\tau A^\alpha \frac{z_{n+1}+z_{n}}{2},A^{-\alpha}\left(\dot{z}_{n+1}+\dot{z}_{n}\right)\right\rangle\right]\\
&~~~~+\mathrm{E}\left[\tau \left\|f(u_{n})\right\|^2\right]+\mathrm{E}\left[\tau \left\|A^{-\alpha}\left(\dot{z}_{n+1}+\dot{z}_{n}\right)\right\|^2\right]\\
&~~~~+\mathrm{E}\left[\tau\left\|\left(f(u_{n})-f(u_{n-1})\right)\right\|^2\right]\\
\end{split}.
\end{equation}
Using Eqs. \eqref{eq:4.0-0-1}, \eqref{eq:4.0-6} and \eqref{eq:4.0-7} leads to
\begin{equation*}
\begin{split}
&\mathrm{E}\left[\left\|A^{-\frac{\alpha}{2}}\dot{z}_{n+1}\right\|^2\right]+\mathrm{E}\left[\left\|z_{n+1}\right\|^2\right]\\
&\lesssim\tau\sum^{n+1}_{i=1}\left(\mathrm{E}\left[\left\|f\left(u_{i}\right)\right\|^2\right]+\mathrm{E}\left[\left\|A^{-\frac{\alpha}{2}}\dot{z}_{i}\right\|^2\right]\right)+\mathrm{E}\left[\left\|A^{-\frac{\alpha}{2}}\dot{z}_{0}\right\|^2\right]+\mathrm{E}\left[\left\|z_{0}\right\|^2\right]\\
&\lesssim\tau\sum^{n+1}_{i=1}\left(\mathrm{E}\left[\left\|z_{i}\right\|^2\right]+\mathrm{E}\left[\left\|A^{-\frac{\alpha}{2}}\dot{z}_{i}\right\|^2\right]\right)+\mathrm{E}\left[\left\|A^{-\frac{\alpha}{2}}\dot{z}_{0}\right\|^2\right]+\mathrm{E}\left[\left\|z_{0}\right\|^2\right]+1\\
\end{split}.
\end{equation*}
According to the discrete Gr\"onwall inequality, we get
\begin{equation*}
\begin{split}
\mathrm{E}\left[\left\|A^{-\frac{\alpha}{2}}\dot{z}_{n+1}\right\|^2\right]+\mathrm{E}\left[\left\|z_{n+1}\right\|^2\right]\lesssim&\mathrm{E}\left[\left\|A^{-\frac{\alpha}{2}}\dot{z}_{0}\right\|^2\right]+\mathrm{E}\left[\left\|z_{0}\right\|^2\right]+1\\
\end{split},
\end{equation*}
which shows
\begin{equation*}
\begin{split}
\mathrm{E}\left[\left\|u_{n+1}\right\|^2\right]\lesssim&\mathrm{E}\left[\left\|A^{-\frac{\alpha}{2}}\dot{z}_{0}\right\|^2\right]+\mathrm{E}\left[\left\|z_{0}\right\|^2\right]+1\\
\end{split}.
\end{equation*}
The proof is completed.
\end{proof}

\begin{theorem}\label{th:4}
Let $\dot{z}_{n+1}$ and $z_{n+1}$ be expressed by Eq. \eqref{eq:4.0-5}. Besides Assumptions \ref{as:2.1}, suppose that $ f'(t)<\infty$ and $f'(t)$ satisfies the Lipschitz condition, $\left\|u_0\right\|_{L^p(D,\dot{U}^{\gamma+\alpha})}<\infty$, $\left\|v_0\right\|_{L^p(D,\dot{U}^{\gamma})}<\infty$, $\varepsilon>0$, $\gamma=\alpha+2\rho-\frac{d+\varepsilon}{2}$, and $\gamma>\alpha$. Then
\begin{equation*}
\left\|u(t_{n+1})-u_{n+1}\right\|_{L^2(D,U)}\lesssim\tau^{1+\min\{\frac{\gamma-\alpha}{\alpha},H\}}\left(\sqrt{\frac{1}{\varepsilon}}+\left\|u_0\right\|_{L^2(D,\dot{U}^{\gamma+\alpha})}+\left\|v_0\right\|_{L^2(D,\dot{U}^{\gamma})}\right).
\end{equation*}
\end{theorem}
\begin{proof}
Using Eqs. \eqref{eq:3.7-3}, \eqref{eq:4.0-0-3} and \eqref{eq:4.0-0-4}, we have
\begin{equation}\label{eq:4.1-9}
\mathrm{E}\left[\left\|u(t_{n+1})-u_{n+1}\right\|^2\right]\lesssim\mathrm{E}\left[\left\|z(t_{n+1})-z_{n+1}\right\|^2\right]+\sqrt{\frac{1}{\varepsilon}}\tau^{\min\{\frac{\gamma}{\alpha},2\}}
\end{equation}
Combining Eqs. \eqref{eq:4.0-0}, \eqref{eq:4.0-5}, H\"older inequality and Proposition \ref{pro:4} leads to
\begin{equation*}
\begin{split}
&\mathrm{E}\left[\left\langle e_{n+1}-e_{n},e_{n+1}+e_{n}\right\rangle\right]\\
&=\mathrm{E}\left[\left\langle\int^{t_{n+1}}_{t_{n}} \left(\dot{z}(s)-\frac{\dot{z}_{n+1}+\dot{z}_{n}}{2}\right)\mathrm{d}s,e_{n+1}+e_{n}\right\rangle\right]\\
&=\mathrm{E}\left[\left\langle\int^{t_{n+1}}_{t_{n}} \left(\dot{z}(s)-\frac{\dot{z}(t_{n+1})+\dot{z}(t_{n})}{2}\right)\mathrm{d}s,e_{n+1}+e_{n}\right\rangle\right]\\
&~~~~+\mathrm{E}\left[\left\langle\int^{t_{n+1}}_{t_{n}}\left(\frac{\dot{z}(t_{n+1})+\dot{z}(t_{n})}{2}-\frac{\dot{z}_{n+1}+\dot{z}_{n}}{2}\right)\mathrm{d}s,e_{n+1}+e_{n}\right\rangle\right]\\
&=\frac{1}{2}\mathrm{E}\left[\left\langle\int^{t_{n+1}}_{t_{n}} \left(\int^{s}_{t_{n}}\ddot{z}(r)\mathrm{d}r-\int^{t_{n+1}}_{s}\ddot{z}(r)\mathrm{d}r\right)\mathrm{d}s,e_{n+1}+e_{n}\right\rangle\right]\\
&~~~~+\mathrm{E}\left[\left\langle\int^{t_{n+1}}_{t_{n}}\frac{\dot{e}_{n+1}+\dot{e}_{n}}{2}\mathrm{d}s,e_{n+1}+e_{n}\right\rangle\right]\\
&=\frac{1}{2}\mathrm{E}\left[\left\langle\int^{t_{n+1}}_{t_{n}} \left(\int^{s}_{t_{n}}\left(\ddot{z}(r)-\ddot{z}(t_n)\right)\mathrm{d}r+\int^{t_{n+1}}_{s}\left(\ddot{z}(t_n)-\ddot{z}(r)\right)\mathrm{d}r\right)\mathrm{d}s,e_{n+1}+e_{n}\right\rangle\right]\\
&~~~~+\mathrm{E}\left[\left\langle\int^{t_{n+1}}_{t_{n}}\frac{\dot{e}_{n+1}+\dot{e}_{n}}{2}\mathrm{d}s,e_{n+1}+e_{n}\right\rangle\right]\\
&\lesssim\mathrm{E}\left[\int^{t_{n+1}}_{t_{n}} \int^{s}_{t_{n}}\tau\left\|\ddot{z}(r)-\ddot{z}(t_n)\right\|^2\mathrm{d}r\mathrm{d}s\right]+\mathrm{E}\left[\int^{t_{n+1}}_{t_{n}}\int^{t_{n+1}}_{s}\tau\left\|\ddot{z}(t_n)-\ddot{z}(r)\right\|^2\mathrm{d}r\mathrm{d}s\right]\\
&~~~~+\mathrm{E}\left[\tau\left\|e_{n+1}+e_{n}\right\|^2\right]+\mathrm{E}\left[\left\langle\int^{t_{n+1}}_{t_{n}}\frac{\dot{e}_{n+1}+\dot{e}_{n}}{2}\mathrm{d}s,e_{n+1}+e_{n}\right\rangle\right]\\
&\lesssim\tau^{3+\min\{\frac{2\gamma-2\alpha}{\alpha},2\}}\left(\frac{1}{\varepsilon}+\left\|u_0\right\|^2_{L^2(D,\dot{U}^{\gamma+\alpha})}+\left\|v_0\right\|^2_{L^2(D,\dot{U}^{\gamma})}\right)\\
&~~~~+\mathrm{E}\left[\tau\left\|e_{n+1}+e_{n}\right\|^2\right]+\mathrm{E}\left[\left\langle\int^{t_{n+1}}_{t_{n}}\frac{\dot{e}_{n+1}+\dot{e}_{n}}{2}\mathrm{d}s,e_{n+1}+e_{n}\right\rangle\right]
\end{split}.
\end{equation*}
For the fourth equality, we use the fact
\begin{equation}\label{eq:4.0-19}
\int^{t_{n+1}}_{t_{n}} \int^{s}_{t_{n}}\mathrm{d}r\mathrm{d}s=\int^{t_{n+1}}_{t_{n}}\int^{t_{n+1}}_{s}\mathrm{d}r\mathrm{d}s.
\end{equation}
Then
\begin{equation}\label{eq:4.0-9}
\begin{split}
&\mathrm{E}\left[\left\|e_{n+1}\right\|^2\right]-\mathrm{E}\left[\left\|e_{n}\right\|^2\right]\\
&\lesssim\tau^{3+\min\{\frac{2\gamma-2\alpha}{\alpha},2\}}\left(\frac{1}{\varepsilon}+\left\|u_0\right\|^2_{L^2(D,\dot{U}^{\gamma+\alpha})}+\left\|v_0\right\|^2_{L^2(D,\dot{U}^{\gamma})}\right)\\
&~~~~+\mathrm{E}\left[\tau\left\|e_{n+1}+e_{n}\right\|^2\right]+\mathrm{E}\left[\left\langle\int^{t_{n+1}}_{t_{n}}\frac{\dot{e}_{n+1}+\dot{e}_{n}}{2}\mathrm{d}s,e_{n+1}+e_{n}\right\rangle\right]
\end{split}.
\end{equation}
Using again Eqs. \eqref{eq:4.0-0}, \eqref{eq:4.0-5} and weak formulation, we have
\begin{equation}\label{eq:4.0-10}
\begin{split}
&\mathrm{E}\left[\left\langle\dot{e}_{n+1}-\dot{e}_{n},A^{-\alpha}\left(\dot{e}_{n+1}+\dot{e}_{n}\right)\right\rangle\right]\\
&=\mathrm{E}\left[\left\langle-\int^{t_{n+1}}_{t_{n}}A^\alpha \left(z(s)-\frac{z_{n+1}+z_{n}}{2}\right)\mathrm{d}s,A^{-\alpha}\left(\dot{e}_{n+1}+\dot{e}_{n}\right)\right\rangle\right]\\
&~~~~+\mathrm{E}\left[\left\langle\int^{t_{n+1}}_{t_{n}}\left(f(u(s))-f(u_{n})\right)\mathrm{d}s-\frac{\tau}{2}\left(f(u_{n})-f(u_{n-1})\right),A^{-\alpha}\left(\dot{e}_{n+1}+\dot{e}_{n}\right)\right\rangle\right]\\
&=J_1+J_2
\end{split}.
\end{equation}
For $J_1$, combining Proposition \ref{pro:4} and Eq. \eqref{eq:4.0-19} leads to
\begin{equation*}
\begin{split}
J_1\\
=&\mathrm{E}\left[\left\langle -\int^{t_{n+1}}_{t_{n}}A^\alpha \left(z(s)-\frac{z(t_{n+1})+z(t_{n})}{2}\right)\mathrm{d}s-\tau A^\alpha\frac{e_{n+1}+e_{n}}{2},A^{-\alpha}\left(\dot{e}_{n+1}+\dot{e}_{n}\right)\right\rangle\right]\\
=&\mathrm{E}\left[\int^{t_{n+1}}_{t_{n}}\int^{s}_{t_{n}}\left\langle A^\frac{\alpha}{2}\frac{\dot{z}(t_{n+1})-\dot{z}(r)}{2},A^{-\frac{\alpha}{2}}\left(\dot{e}_{n+1}+\dot{e}_{n}\right)\right\rangle\mathrm{d}r\mathrm{d}s-\tau\left\langle\frac{e_{n+1}+e_{n}}{2},\dot{e}_{n+1}+\dot{e}_{n}\right\rangle\right]\\
&+\mathrm{E}\left[\int^{t_{n+1}}_{t_{n}} \int^{t_{n+1}}_{s}\left\langle A^\frac{\alpha}{2}\frac{\dot{z}(r)-\dot{z}(t_{n+1})}{2},A^{-\frac{\alpha}{2}}\left(\dot{e}_{n+1}+\dot{e}_{n}\right)\right\rangle\mathrm{d}r\mathrm{d}s\right]\\
\lesssim&\mathrm{E}\left[\int^{t_{n+1}}_{t_{n}}\int^{s}_{t_{n}}\left\| A^\frac{\alpha}{2}\frac{\dot{z}(t_{n+1})-\dot{z}(r)}{2}\right\|\left\|A^{-\frac{\alpha}{2}}\left(\dot{e}_{n+1}+\dot{e}_{n}\right)\right\|\mathrm{d}r\mathrm{d}s-\tau\left\langle\frac{e_{n+1}+e_{n}}{2},\dot{e}_{n+1}+\dot{e}_{n}\right\rangle\right]\\
&+\mathrm{E}\left[\int^{t_{n+1}}_{t_{n}} \int^{t_{n+1}}_{s}\left\| A^\frac{\alpha}{2}\frac{\dot{z}(r)-\dot{z}(t_{n+1})}{2}\right\|\left\|A^{-\frac{\alpha}{2}}\left(\dot{e}_{n+1}+\dot{e}_{n}\right)\right\|\mathrm{d}r\mathrm{d}s\right]\\
\lesssim&\mathrm{E}\left[\tau\int^{t_{n+1}}_{t_{n}}\int^{s}_{t_{n}}\left\| A^\frac{\alpha}{2}\frac{\dot{z}(t_{n+1})-\dot{z}(r)}{2}\right\|^2\mathrm{d}r\mathrm{d}s-\tau\left\langle\frac{e_{n+1}+e_{n}}{2},\dot{e}_{n+1}+\dot{e}_{n}\right\rangle\right]\\
&+\mathrm{E}\left[\tau\int^{t_{n+1}}_{t_{n}} \int^{t_{n+1}}_{s}\left\| A^\frac{\alpha}{2}\frac{\dot{z}(r)-\dot{z}(t_{n+1})}{2}\right\|^2\mathrm{d}r\mathrm{d}s+\tau\left\|A^{-\frac{\alpha}{2}}\dot{e}_{n+1}\right\|^2+\tau\left\|A^{-\frac{\alpha}{2}}\dot{e}_{n}\right\|^2\right]\\
\lesssim&\mathrm{E}\left[\tau^{3+\min\{\frac{2\gamma-2\alpha}{\alpha},2\}}\left(\frac{1}{\varepsilon}+\left\|u_0\right\|^2_{L^2(D,\dot{U}^{\gamma+\alpha})}+\left\|v_0\right\|^2_{L^2(D,\dot{U}^{\gamma})}\right)\right]\\
&-\mathrm{E}\left[\tau\left\langle\frac{e_{n+1}+e_{n}}{2},\dot{e}_{n+1}+\dot{e}_{n}\right\rangle+\tau\left\|A^{-\frac{\alpha}{2}}\dot{e}_{n+1}\right\|^2+\tau\left\|A^{-\frac{\alpha}{2}}\dot{e}_{n}\right\|^2\right]
\end{split}.
\end{equation*}
Let $0<\theta_1,\theta_2,\theta_3<1$, and
\begin{equation*}
\left\{
\begin{array}{ll}
\Psi_1u(t_{n})=(1-\theta_1)u(s)+\theta_1u(t_{n}),\quad t_{n}<s<t_{n+1}\\
\Psi_2u(t_{n})=(1-\theta_2)u(t_{n+1})+\theta_2u(t_{n})
\end{array}
\right..
\end{equation*}
Using Lagrange mean value theorem and Eq. \eqref{eq:4.1-9}, we have
\begin{equation*}
\begin{split}
J_2&=\mathrm{E}\left[\left\langle\int^{t_{n+1}}_{t_{n}}\left(f(u(s))-f(u(t_{n}))\right)\mathrm{d}s-\frac{\tau}{2}\left(f(u(t_{n}))-f(u(t_{n-1}))\right),A^{-\alpha}\left(\dot{e}_{n+1}+\dot{e}_{n}\right)\right\rangle\right]\\
&~~~~+\mathrm{E}\left[\left\langle\frac{3\tau}{2}\left(f(u(t_{n}))-f(u_{n})\right)-\frac{\tau}{2}\left(f(u(t_{n-1}))-f(u_{n-1})\right),A^{-\alpha}\left(\dot{e}_{n+1}+\dot{e}_{n}\right)\right\rangle\right]\\
&=\mathrm{E}\left[\left\langle\int^{t_{n+1}}_{t_{n}}\int^{s}_{t_{n}}f'\left(\Psi_1u(t_{n})\right)v(r)\mathrm{d}r\mathrm{d}s,A^{-\alpha}\left(\dot{e}_{n+1}+\dot{e}_{n}\right)\right\rangle\right]\\
&~~~~-\mathrm{E}\left[\left\langle\int^{t_{n+1}}_{t_{n}}\int^{s}_{t_{n}}f'\left(\Psi_2u(t_{n-1})\right)v\left((1-\theta_3)t_{n}+\theta_3t_{n-1}\right)\mathrm{d}r\mathrm{d}s,A^{-\alpha}\left(\dot{e}_{n+1}+\dot{e}_{n}\right)\right\rangle\right]\\
&~~~~+\mathrm{E}\left[\left\langle\frac{3\tau}{2}\left(f(u(t_{n}))-f(u_{n})\right)-\frac{\tau}{2}\left(f(u(t_{n-1}))-f(u_{n-1})\right),A^{-\alpha}\left(\dot{e}_{n+1}+\dot{e}_{n}\right)\right\rangle\right]\\
&\lesssim \mathrm{E}\left[\left\langle\int^{t_{n+1}}_{t_{n}}\int^{s}_{t_{n}}f'\left(\Psi_1u(t_{n})\right)v(r)\mathrm{d}r\mathrm{d}s,A^{-\alpha}\left(\dot{e}_{n+1}+\dot{e}_{n}\right)\right\rangle\right]\\
&~~~~-\mathrm{E}\left[\left\langle\int^{t_{n+1}}_{t_{n}}\int^{s}_{t_{n}}f'\left(\Psi_2u(t_{n-1})\right)v\left((1-\theta_3)t_{n}+\theta_3t_{n-1}\right)\mathrm{d}r\mathrm{d}s,A^{-\alpha}\left(\dot{e}_{n+1}+\dot{e}_{n}\right)\right\rangle\right]\\
&~~~~+\mathrm{E}\left[\tau\left\|e_{n}\right\|^2+\tau\left\|e_{n-1}\right\|^2+\tau\left\|A^{-\frac{\alpha}{2}}\dot{e}_{n+1}\right\|^2+\tau\left\|A^{-\frac{\alpha}{2}}\dot{e}_{n}\right\|^2\right]+\frac{1}{\varepsilon}\tau^{1+\min\{\frac{2\gamma}{\alpha},4\}}\\
&=III+\mathrm{E}\left[\tau\left\|e_{n}\right\|^2+\tau\left\|e_{n-1}\right\|^2+\tau\left\|A^{-\frac{\alpha}{2}}\dot{e}_{n+1}\right\|^2+\tau\left\|A^{-\frac{\alpha}{2}}\dot{e}_{n}\right\|^2\right]+\frac{1}{\varepsilon}\tau^{1+\min\{\frac{2\gamma}{\alpha},4\}}
\end{split}.
\end{equation*}
Lemma \ref{le:1} implies
\begin{equation*}
\begin{split}
III&\lesssim\mathrm{E}\left[\int^{t_{n+1}}_{t_{n}}\int^{s}_{t_{n}}\left\|\left(|\Psi_1u(t_{n})-\Psi_2u(t_{n-1})|\right)v(r)\right\|\left\|A^{-\alpha}\left(\dot{e}_{n+1}+\dot{e}_{n}\right)\right\|\mathrm{d}r\mathrm{d}s\right]\\
&~~~~+\mathrm{E}\left[\int^{t_{n+1}}_{t_{n}}\int^{s}_{t_{n}}\left\|v(r)-v\left((1-\theta_3)t_{n}+\theta_3t_{n-1}\right)\right\|\left\|A^{-\alpha}\left(\dot{e}_{n+1}+\dot{e}_{n}\right)\right\|\mathrm{d}r\mathrm{d}s\right]\\
&\lesssim\mathrm{E}\left[\int^{t_{n+1}}_{t_{n}}\int^{s}_{t_{n}}\left\|\int^{s}_{t_{n}}v(r_1)\mathrm{d}r_1v(r)\right\|\left\|A^{-\alpha}\left(\dot{e}_{n+1}+\dot{e}_{n}\right)\right\|\mathrm{d}r\mathrm{d}s\right]\\
&~~~~+\mathrm{E}\left[\int^{t_{n+1}}_{t_{n}}\int^{s}_{t_{n}}\left\|\int^{t_{n}}_{t_{n-1}}v(r_1)\mathrm{d}r_1v(r)\right\|\left\|A^{-\alpha}\left(\dot{e}_{n+1}+\dot{e}_{n}\right)\right\|\mathrm{d}r\mathrm{d}s\right]\\
&~~~~+\mathrm{E}\left[\int^{t_{n+1}}_{t_{n}}\int^{s}_{t_{n}}\left\|v(r)-v\left((1-\theta_3)t_{n}+\theta_3t_{n-1}\right)\right\|\left\|A^{-\alpha}\left(\dot{e}_{n+1}+\dot{e}_{n}\right)\right\|\mathrm{d}r\mathrm{d}s\right]\\
&\lesssim\mathrm{E}\left[\tau^5\left(\left\|u_0\right\|^4+\left\|v_0\right\|^4+1\right)+\tau^{3+\min\{2\frac{\gamma-\alpha}{\alpha},2H\}}\left(\left\|u_0\right\|^2+\left\|v_0\right\|^2+\frac{1}{\varepsilon}\right)\right]\\
&~~~~+\mathrm{E}\left[\tau\left\|A^{-\frac{\alpha}{2}}\dot{e}_{n+1}\right\|^2+\tau\left\|A^{-\frac{\alpha}{2}}\dot{e}_{n}\right\|^2\right].
\end{split}.
\end{equation*}
Combining $J_1$, $J_2$, Eqs. \eqref{eq:4.0-9} and \eqref{eq:4.0-10} leads to
\begin{equation*}
\begin{split}
&\mathrm{E}\left[\left\|e_{n+1}\right\|^2\right]-\mathrm{E}\left[\left\|e_{n}\right\|^2\right]+\mathrm{E}\left[\left\|A^{-\frac{\alpha}{2}}\dot{e}_{n+1}\right\|^2\right]-\mathrm{E}\left[\left\|A^{-\frac{\alpha}{2}}\dot{e}_{n}\right\|^2\right]\\
&\lesssim\tau^{3+\min\{\frac{2\gamma-2\alpha}{\alpha},2H\}}\left(\frac{1}{\varepsilon}+\left\|u_0\right\|^2_{L^2(D,\dot{U}^{\gamma+\alpha})}+\left\|v_0\right\|^2_{L^2(D,\dot{U}^{\gamma})}\right)\\
&~~~~+\tau\mathrm{E}\left[\left\|e_{n}\right\|^2+\left\|e_{n-1}\right\|^2+\left\|A^{-\frac{\alpha}{2}}\dot{e}_{n+1}\right\|^2+\left\|A^{-\frac{\alpha}{2}}\dot{e}_{n}\right\|^2\right]
\end{split},
\end{equation*}
which implies
\begin{equation*}
\begin{split}
&\mathrm{E}\left[\left\|e_{n+1}\right\|^2\right]+\mathrm{E}\left[\left\|A^{-\frac{\alpha}{2}}\dot{e}_{n+1}\right\|^2\right]\\
&\lesssim\tau^{2+\min\{\frac{2\gamma-2\alpha}{\alpha},2H\}}\left(\frac{1}{\varepsilon}+\left\|u_0\right\|^2_{L^2(D,\dot{U}^{\gamma+\alpha})}+\left\|v_0\right\|^2_{L^2(D,\dot{U}^{\gamma})}\right)\\
\end{split}.
\end{equation*}
Then, using above equation and Eq. \eqref{eq:4.1-9}, we have
\begin{equation*}
\begin{split}
\mathrm{E}\left[\left\|u(t_{n+1})-u_{n+1}\right\|^2\right]&\lesssim\tau^{2+\min\{\frac{2\gamma-2\alpha}{\alpha},2H\}}\left(\frac{1}{\varepsilon}+\left\|u_0\right\|^2_{L^2(D,\dot{U}^{\gamma+\alpha})}+\left\|v_0\right\|^2_{L^2(D,\dot{U}^{\gamma})}\right)
\end{split}.
\end{equation*}
\end{proof}

\section{Numerical Experiments}\label{sec:6}

In order to verify the theoretical results, we perform several numerical examples in this section. All numerical errors are given in the sense of mean-squared $L^2$-norm.

We solve \eqref{eq:1.05} in the one-dimensional domain $D=(0,1)$ by the proposed schemes \eqref{eq:4.0-1} and \eqref{eq:4.0-5} with the smooth initial data $u_0=\frac{1}{\sqrt{2}}\sin(2\pi x)$, $v_0=\frac{1}{2\sqrt{2}}\sin(3\pi x)$, the time mesh size $\tau=\frac{T}{N}$ and $f(u(t))=\sin(u(t))$. And using the modified spectral Galerkin approximation \cite{LiuDeng} discretize the problem \eqref{eq:1.05} in space with $\{\lambda_{i},\phi_{i}\}^M_{i=1}$, which are Dirichlet eigenpairs of Laplacian in $D=(0,1)$. Let $u_N$ be the numerical solution of the fully discrete scheme. To calculate the time convergence orders, the following formulas are used
\begin{equation*}
 \frac{\ln\left(\left\|u_{aN}-u_{N}\right\|_{L^2(D,U)}/
\left\|u_{N}-u_{N/a}\right\|_{L^2(D,U)}\right)}{\ln a},
\end{equation*}
where the constant $a>1$.
In numerical simulations, $1000$ samples are used for the approximation of the expected
values $\mathrm{E}\left[\left\|u_{aN}-u_{N}\right\|^2\right]$. For each sample, we generate $M$ independent fractional Brownian motions $\beta^i_H(t)$, $i=1,2,\dots,M$.

\subsection{Low order case}
As $0<\gamma\le\alpha$, we use Eqs. \eqref{eq:4.0-1} and \eqref{eq:4.0-0-1} to discretize the problem \eqref{eq:1.05} in time. Choose $\rho=0.25$ and three values of $\alpha\in(0,1]$, i.e., $\alpha=0.6,0.8,1$. Then, Theorem \ref{th:2} implies the convergence rates are close to $1$. Taking $M=1600$ guarantee that the dominant errors arise from the temporal approximation. For $H=0.8$, the following numerical results confirm the error estimates in Theorem \ref{th:2}.
\begin{table}[htp]
\renewcommand\arraystretch{1.6}
\caption{Time convergence rates with $M=1600$, $T=0.5$, $H=0.8$, and $\rho=0.25$.}\label{table:1}
\centering
\begin{tabular}{c c c c c c c c }
\hline
$N$ & $\alpha=0.6$ & Rate &$\alpha=0.8$ &Rate & $\alpha=1$ &Rate  \\
\hline
  32&1.826e-03&     & 1.028e-02&     &3.440e-02&          \\
  64&9.441e-04& 0.952& 5.051e-03& 1.025&1.828e-02&0.912\\
  128&4.829e-04& 0.967& 2.500e-03& 1.015&9.416e-03&0.957  \\
  \hline
\end{tabular}
\end{table}

\subsection{Higher order case}
For $\gamma>\alpha$, using Eqs. \eqref{eq:4.0-5} and \eqref{eq:4.0-0-3} solves numerically the problem \eqref{eq:1.05}. The simulation of the approximation \eqref{eq:stochastic} for the stochastic integral can be found in \cite{LiuDeng}.

%\begin{table}[H]
%\renewcommand\arraystretch{1.6}
%\caption{Time convergence rates with $N=2000$, $T=1$, $H=0.9$, and $\alpha=0.4$.}\label{table:2}
%\centering
%\begin{tabular}{c c c c c c c c }
%\hline
%$M$ & $\rho=0.35$ & Rate &$\rho=0.4$ &Rate & $\rho=0.45$ &Rate  \\
%\hline
%  4&1.796e-03&     & 1.027e-02&     &9.757e-02&          \\
%  8&9.410e-04& 0.933& 5.051e-03& 1.024&4.274e-02&1.191\\
%  16&4.829e-04& 0.963& 2.500e-03& 1.015&9.416e-03&0.958  \\
%  \hline
%\end{tabular}
%\end{table}
% As $\alpha<\gamma<H\alpha+\alpha$, the theoretical convergence rate
%is approximately $1+\frac{\gamma-\alpha}{\alpha}$. Table \ref{table:2} shows that the numerical results coincide with theoretical estimates.

\begin{figure}[htb]
  \centering
  \label{fig:a}\includegraphics[scale=0.6]{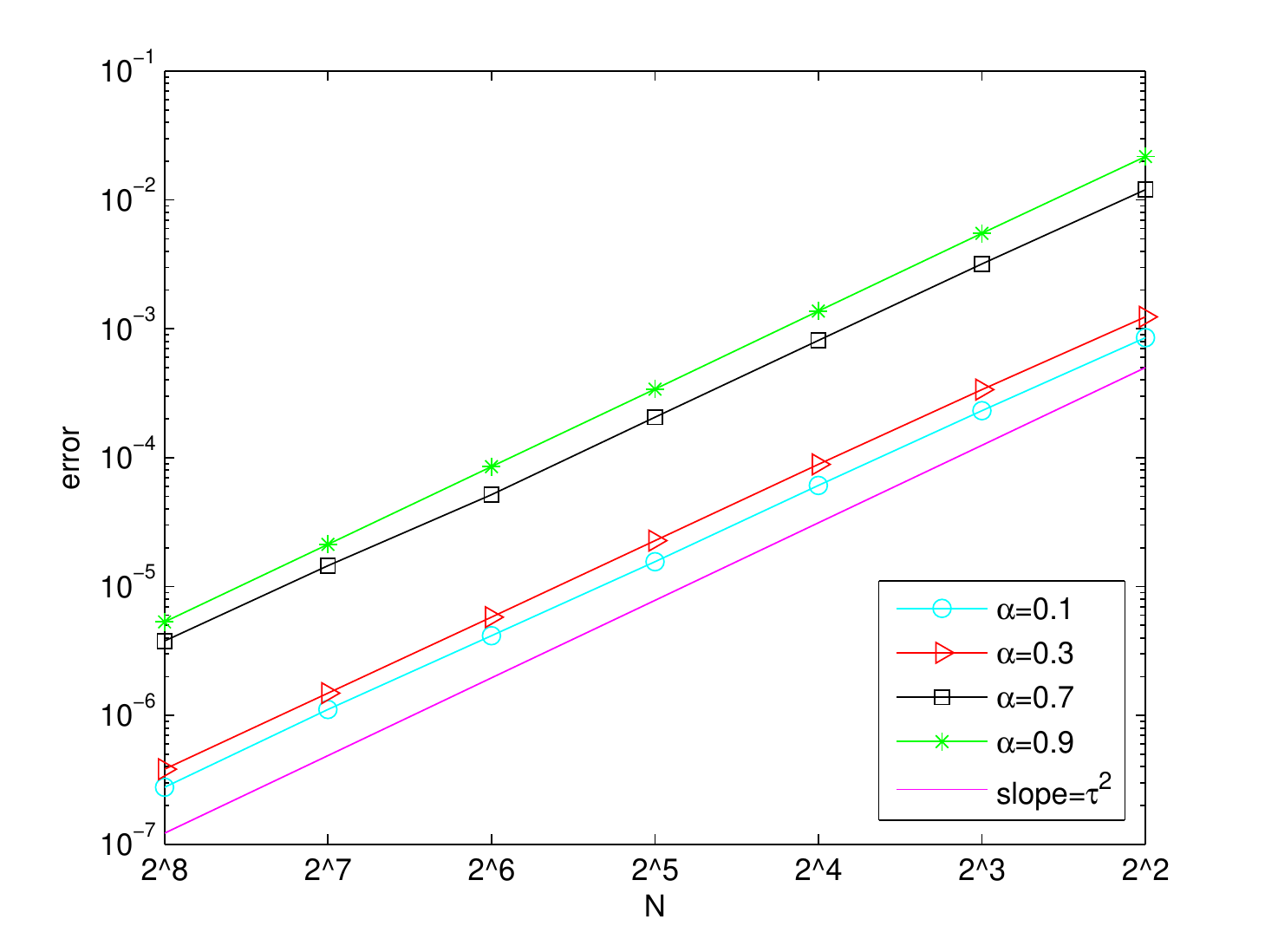}
  \caption{Temporal error convergence of the modified stochastic trigonometric method for the space-time fractional Gaussian noise $(H=0.6)$.}
  \label{fig:fig1}
\end{figure}

The numerical results are presented in Figures \ref{fig:fig1} and \ref{fig:fig2} with $T=0.5$, $\rho=1.5$ and $M=200$, which ensures the temporal error is the dominant one.  Figures \ref{fig:fig1} and \ref{fig:fig2} show that the temporal convergence rates have at least an order of $1+\min\{\frac{\gamma-\alpha}{\alpha},H\}$ by using the proposed scheme, which in turn justifies error estimates of Theorem \ref{th:4}.

\begin{figure}[htb]
  \centering
  \label{fig:b}\includegraphics[scale=0.6]{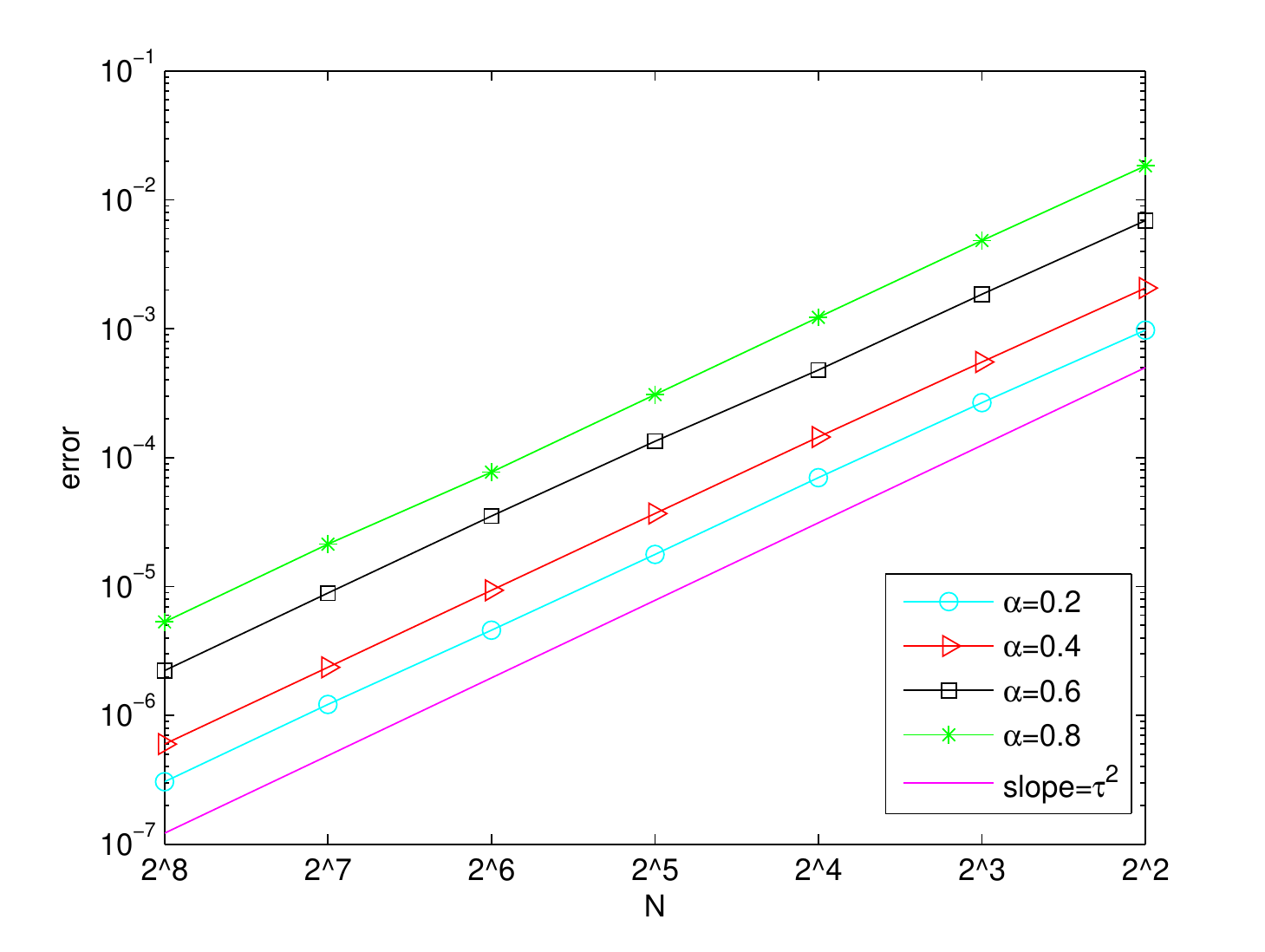}
  \caption{Temporal error convergence of the modified stochastic trigonometric method for the space-time fractional Gaussian noise $(H=0.8)$.}
  \label{fig:fig2}
\end{figure}

\section{Conclusion}\label{sec:7}

This paper discusses the time discretization of the equation describing the wave propagation with additive fractional Gaussian noise and the corresponding error analyses. Based on the temporal regularity of $u(t)$, we design two numerical difference methods to discretize the problem \eqref{eq:1.05}. When the solution is H\"older continuous in the sense of mean-squared $L^2$-norm $(0<\gamma\le\alpha)$, using rectangle formula scheme leads to the low order time discretization which has order $\frac{\gamma}{\alpha}$. If the first order derivative of the solution with respect to time is H\"older continuous, a higher order time discretization is proposed by using trapezoidal rule. Then the superlinear convergence is obtained under the mean-squared $L^2$-norm.

\noindent {\bf Acknowledgments.} The author gratefully thank the anonymous referees for valuable comments and suggestions
in improving this paper.

\end{document}